\documentclass[12pt,english]{amsart}

\usepackage[T1]{fontenc}
\usepackage[utf8]{inputenc}
\usepackage[main=english]{babel}

\usepackage{geometry}
\usepackage{amsmath,amsthm,amssymb,amsfonts,mathtools}
\usepackage{mathrsfs}
\usepackage{esint}

\usepackage{graphicx}
\usepackage{float}
\usepackage{tikz}

\usepackage{comment}
\usepackage{indentfirst}
\usepackage[hidelinks]{hyperref}

\theoremstyle{definition}
\newtheorem{theorem}{Theorem}[section]
\newtheorem{lemma}[theorem]{Lemma}
\newtheorem{remark}[theorem]{Remark}

\newtheorem{proposition}[theorem]{Proposition}
\newtheorem{conjecture}[theorem]{Conjecture}

\newcommand{\Ztwo}{\mathbb Z_2}
\newcommand{\cZ}{\mathcal Z}
\newcommand{\cC}{\mathcal C}
\newcommand{\cH}{\mathcal H}
\newcommand{\Mmass}{\mathbf M}

\title{On filtration formula and nodal sets in min-max theory}
\author{Talant Talipov}
\date{}

\begin{document}

\begin{abstract}
Let $(M^{n+1},g)$ be a closed smooth Riemannian manifold, and let
\[
0=\lambda_0\leq\lambda_1\leq\ldots
\]
be the spectrum of the Laplace--Beltrami operator, with corresponding
real eigenfunctions $\{\phi_j\}_{j\geq0}$.  Let $\{\omega_p(M,g)\}_{p\geq1}$ denote
Gromov's volume spectrum.
In 2014, Marques--Neves conjectured that the sweepout generated by
the nodal sets of $\phi_0,\ldots,\phi_p$ is asymptotically optimal for $\omega_p(M,g)$ as $p\to\infty$.
We show that this conjecture is false in every dimension, even for
positively curved real-analytic metrics.

Nevertheless, we show that a less rigid notion of nodal sweepouts recovers the volume spectrum.  Let
\[
E_N(M,g):=\operatorname{span}\{\phi_0,\ldots,\phi_N\},
\qquad
a_N\in H^1\bigl(\mathbb P(E_N(g));\mathbb Z_2\bigr)
\]
be the generator, and for $p\geq1$ define
\[
\mathcal N_p^N(g)
:=
\left\{
\Psi \mid
X\text{ is a finite complex},\
\Psi\in C\bigl(X,\mathbb P(E_N(g))\bigr),\
\Psi^*(a_N^p)\neq0
\right\}.
\]
We define the nodal $(p,N)$-width by
\[
\nu_p^N(M,g)
:=
\inf_{\Psi\in\mathcal N_p^N(g)}
\sup_{x\in\operatorname{dmn}(\Psi)}
\mathcal H_g^{n}\bigl(\{\Psi_x=0\}\bigr).
\]
We prove the following filtration formula
\[
\omega_p(M,g)
=
\lim_{N\to\infty}\nu_p^N(M,g)
\]
for every $p\geq1$. We state several questions on nodal geometry and min-max theory.
\end{abstract}

\maketitle

\section{Introduction}

Let $(M^{n+1},g)$ be a connected closed smooth Riemannian manifold.  The
Laplace--Beltrami operator has a discrete spectrum
\[
0=\lambda_0\leq\lambda_1\leq\lambda_2\leq\ldots,
\]
where the eigenvalues are repeated with multiplicity.  We fix an
$L^2(M,g)$-orthonormal real eigenbasis
$\{\phi_j\}_{j\geq0}$, so that
\[
-\Delta_g\phi_j=\lambda_j\phi_j.
\]

The Laplace spectrum admits a variational characterization.  If
\[
\mathcal R_g(f)
:=
\frac{\displaystyle\int_M|\nabla f|_g^2\,dV_g}
{\displaystyle\int_M f^2\,dV_g}
\]
denotes the Rayleigh quotient, then
\[
\lambda_p
=
\inf_{\substack{V\subset H^1(M)\\ \dim V=p+1}}
\ \sup_{0\neq f\in V}
\mathcal R_g(f).
\]

Gromov \cite{Gromov} introduced a nonlinear analogue of the Laplace--Beltrami spectrum. 
Let
\[
\mathcal B(M;\mathbb Z_2)
:=
\left\{
\partial\Omega:
\Omega\in\mathcal C(M)
\right\}
\]
denote the space of boundaries of Caccioppoli sets in $M$, endowed
with the flat topology. Almgren's isomorphism theorem \cite{Almgren} implies that
$\mathcal B(M;\mathbb Z_2)$
is weakly homotopy equivalent to $\mathbb{RP}^\infty$. In particular,
\[
H^*\bigl(\mathcal B(M;\mathbb Z_2);\mathbb Z_2\bigr)
\simeq
\mathbb Z_2[\bar\lambda],
\]
where $\bar\lambda$ is the generator. For $p\geq1$, define
\[
\mathcal P_p(M)
:=
\left\{
\Phi \mid
X^p\text{ is a finite complex},\
\Phi\in C\bigl(X,\mathcal B(M;\mathbb Z_2)\bigr),\
\Phi^*(\bar\lambda^p)\neq0
\right\}.
\]
Elements of $\mathcal P_p(M)$ are called \textit{$p$-sweepouts}.
The \textit{$p$-width} of $(M,g)$ is
\[
\omega_p(M,g)
:=
\inf_{\Phi\in\mathcal P_p(M)}
\sup_{x\in\operatorname{dmn}(\Phi)}
\mathbf M_g\bigl(\Phi(x)\bigr).
\]
The sequence
$\{\omega_p(M,g)\}_{p\geq1}$ is called the \emph{volume
spectrum}.

The regularity theory
of Almgren \cite{Almgren1}, Pitts \cite{Pitts} and Schoen--Simon \cite{SchoenSimon}
implies that, in
ambient dimensions $3\leq n+1\leq7$, each $\omega_p(M,g)$ is the mass of
a stationary integral varifold of the form
\begin{equation}
\label{eq:minmax-realization}
V_p=\sum_{j=1}^k m_j|\Sigma_j|,
\end{equation}
where the $\Sigma_j$ are closed smooth embedded minimal hypersurfaces
and $m_j$ are positive integers. In higher dimensions the same
statement holds with optimal regularity, namely, with a singular set of
Hausdorff dimension at most $n-7$.

The volume spectrum has played a central role in the resolution of
Yau's abundance conjecture \cite{Yau}.  Marques--Neves \cite{MN} first
proved the existence of infinitely many closed embedded
minimal hypersurfaces on manifolds of positive Ricci curvature in dimensions
$3\leq n+1\leq7$.
Irie--Marques--Neves \cite{IMN} proved the
conjecture for $C^\infty$-generic metrics in dimensions
$3\leq n+1\leq7$. Finally, Song \cite{Song} proved
the conjecture for every closed Riemannian manifold of dimension
$3\leq n+1\leq7$.  In ambient dimensions $n+1\geq8$,
Li \cite{LiHigher} proved that a closed manifold equipped with a
$C^\infty$-generic metric contains infinitely many minimal hypersurfaces
with optimal regularity. Moreover, Li--Wang \cite{LiWang}
proved that every closed $8$-manifold with a $C^\infty$-generic metric
contains infinitely many smooth closed embedded minimal hypersurfaces 
(see also Chodosh--Liokumovich--Spolaor \cite{CLS} and Wang \cite{WangDeformations}).

The analogy between spectral geometry and min-max theory is supported by a series of results. The classical
Weyl law \cite{Weyl} gives
\[
\lambda_p
\sim
4\pi^2
\left(
\frac{p}{\mu_{n+1}\operatorname{Vol}_g(M)}
\right)^{2/(n+1)},
\]
where $\mu_{n+1}$ is the volume of the Euclidean unit $(n+1)$-ball.
First, Gromov \cite{GromovWaist}, Guth \cite{Guth} and Marques--Neves \cite{MN} established the Weyl-type
growth order of the volume spectrum:
\[
c(M,g)\,p^{1/(n+1)}\leq \omega_p(M,g)\leq C(M,g)\,p^{1/(n+1)}.
\]
Liokumovich--Marques--Neves \cite{LMN} proved the corresponding
Weyl law for the volume spectrum
\[
\omega_p(M,g)
\sim
a_n\,\operatorname{Vol}_g(M)^{n/(n+1)}p^{1/(n+1)},
\]
for a universal dimensional constant $a_{n}>0$.

Uhlenbeck \cite{Uhlenbeck} proved that for a generic metric all
Laplace eigenvalues are simple.  On the min-max side, Zhou's
multiplicity one theorem \cite{Zhou} (see also Chu--Li's strong multiplicity one theorem \cite{ChuLiMultiplicity}) states that, for a bumpy metric
and $3\leq n+1\leq7$, the weights in \eqref{eq:minmax-realization} for
$\omega_p(M,g)$ can be chosen with
\[
m_1=\ldots=m_k=1.
\]

If $\lambda_p$ is simple, then $\phi_p$ has Morse index $p$ as a critical
point of the Rayleigh quotient on the $L^2$-unit sphere.  Analogously,
Marques--Neves \cite{MNIndex} proved that, for
$3\leq n+1\leq7$, the realization \eqref{eq:minmax-realization} can be
chosen so that
\[
\sum_{j=1}^k \operatorname{index}(\Sigma_j)\leq p.
\]
Zhou \cite{Zhou} strengthened this to the weighted bound
\[
\sum_{\Sigma_j\ {\rm orientable}}
m_j\,\operatorname{index}(\Sigma_j)
+
\sum_{\Sigma_j\ {\rm nonorientable}}
\frac{m_j}{2}\,\operatorname{index}(\Sigma_j)
\leq p.
\]
Li \cite{LiIndex} proved the corresponding index bound in arbitrary
dimension for min-max minimal hypersurfaces with optimal regularity.
Moreover, for a generic metric and $3\leq n+1\leq7$,
Marques--Neves \cite{MNExactIndex} proved that the realization
\eqref{eq:minmax-realization} can be chosen so that
\[
m_1=\ldots=m_k=1,
\qquad
\sum_{j=1}^k \operatorname{index}(\Sigma_j)=p.
\]

The nodal set of a non-constant eigenfunction $\phi_j$ is
$C(M,g)\lambda_j^{-1/2}$-dense in $M$. On the minimal-hypersurfaces theory side, Irie--Marques--Neves
\cite{IMN} proved that, for a generic metric and $3\leq n+1\leq7$, the
union of all closed embedded minimal hypersurfaces is dense in $M$.
Moreover, Marques--Neves--Song \cite{MNS} proved that, for a
generic metric and $3\leq n+1\leq7$, there exists a sequence of
closed embedded minimal hypersurfaces which equidistributes in $M$. A similar equidistribution result about nodal sets of Laplace--Beltrami eigenfunctions
was proved by Han \cite{HanNodal}.

Motivated by these analogies, Marques--Neves \cite{MN,NevesICM,MarquesICM} proposed connecting the two
theories more directly through nodal sets of Laplace eigenfunctions. The nodal geometry of Laplace--Beltrami eigenfunctions has an extensive literature; see, among many others,
\cite{Cheng,Bruning,Nadirashvili,DonnellyFefferman,
DonnellyFeffermanSurfaces,HardtSimonNodal,DongNodal,
ColdingMinicozzi,SoggeZelditch,SoggeZelditchII,
Steinerberger,LogunovUpper,LogunovLower,
LogunovMalinnikovaNadirashviliNazarov}
and the surveys \cite{ZelditchSurvey,LogunovMalinnikovaReview}.
Nodal sets of finite sums of Laplace--Beltrami
eigenfunctions have also been studied in
\cite{Lin91,DonnellySums,JerisonLebeau,BBKH,Decio,NazarovSodin,BeliaevWigman,SarnakWigman,CanzaniSarnak,SartoriWigman,WigmanSurvey}. For
\[
E_N(g):=\operatorname{span}\{\phi_0,\ldots,\phi_N\},
\]
consider the family
\begin{equation}
\label{eq:MN-nodal-sweepout}
\begin{split}
\mathbb{RP}^p
&\longrightarrow
\mathcal B(M;\mathbb Z_2),\\
[\theta_0:\ldots:\theta_p]
&\longmapsto
\partial
\left\{
\theta_0\phi_0+\ldots+\theta_p\phi_p<0
\right\}.
\end{split}
\end{equation}
Set
\[
\Phi_p(M,g)
:=
\sup_{0\neq f\in E_p(g)}
\Mmass_g\bigl(\partial\{f<0\}\bigr).
\]
Beck--Becker-Kahn--Hanin \cite{BBKH} proved that
\eqref{eq:MN-nodal-sweepout} is indeed an admissible $p$-sweepout. Thus,
\[
\omega_p(M,g)\leq\Phi_p(M,g).
\]
Marques--Neves \cite{MN,NevesICM,MarquesICM} posed the following conjecture.

\begin{conjecture}[Marques--Neves' asymptotic optimality conjecture, \cite{MN,NevesICM,MarquesICM}]
\label{conj:asymptotic-optimality}
As $p\to\infty,$
\[
\frac{\Phi_p(M,g)}{\omega_p(M,g)}
\longrightarrow 1.
\]
\end{conjecture}

We show that Conjecture \ref{conj:asymptotic-optimality} fails in every ambient dimension $n+1 \geq 2$, even for positively curved real-analytic metrics.

\begin{theorem}
\label{prop:counterexample-asymptotic-optimality}
For every $n\geq 1$ there exists $0<\varepsilon<1$ such that the ellipsoid
\[
M_\varepsilon
=
\left\{
(x,z)\in\mathbb R^{n+1}\times\mathbb R:
|x|^2+\frac{z^2}{\varepsilon^2}=1
\right\},
\]
with the induced metric $g_\varepsilon$, satisfies
\[
\limsup_{p\to\infty}
\frac{\Phi_p(M_\varepsilon,g_\varepsilon)}
{\omega_p(M_\varepsilon,g_\varepsilon)}>1.
\]
\end{theorem}

Beck--Becker-Kahn--Hanin \cite{BBKH} asked whether the
inequality
\[
\Phi_p(M,g)
\leq
\sup_{f\in E_p(g)\setminus\{0\}}
\cH_g^{n}(\{f=0\})
\]
can be strict.  As an ingredient in the proof of
Theorem~\ref{prop:counterexample-asymptotic-optimality}, we answer this
question negatively (see the definition of $\nu_p^N(M,g)$ below).

\begin{proposition}
\label{prop:bbkh-equality}
For every $p\geq1$,
\[
\Phi_p(M,g)
=
\sup_{f\in E_p(g)\setminus\{0\}}
\cH_g^{n}(\{f=0\})
=
\nu_p^p(M,g).
\]
\end{proposition}

The preceding construction concerns the rigid choice of the first
$p+1$ eigenfunctions.  To formulate a less rigid notion of a nodal
sweepout, for $N\geq1$ let
\[
a_N\in H^1\bigl(\mathbb P(E_N(g));\mathbb Z_2\bigr)
\]
denote the generator.  For $p\geq1$ and $N\geq p$, define
\[
\mathcal N_p^N(g)
:=
\left\{
\Psi \mid
X^p\text{ is a finite complex},\
\Psi\in C\bigl(X,\mathbb P(E_N(g))\bigr),\
\Psi^*(a_N^p)\neq0
\right\}.
\]
We define the \textit{nodal $(p,N)$-width} by
\[
\nu_p^N(M,g):=
\inf_{\Psi\in\mathcal N_p^N(g)}
\sup_{x\in \operatorname{dmn}(\Psi)}
\cH_g^{n}\bigl(\{\Psi_x=0\}\bigr).
\]

Using this less rigid notion, we show that nodal sweepouts recover the volume spectrum. The main result of the paper is the following filtration formula.

\begin{theorem}
\label{thm:filtration}
For every $p\geq1$,
\[
\omega_p(M,g)
=
\lim_{N\to\infty}
\nu_p^N(M,g).
\]
\end{theorem}

\begin{remark}
\label{rem:general-filtration}
The Laplace filtration is not essential in Theorem~\ref{thm:filtration}.
More generally, let
$V_1\subset V_2\subset\ldots\subset C^\infty(M)$
be finite-dimensional linear subspaces such that
\[
\overline{\bigcup_{N\geq1}V_N}^{\,C^\infty}=C^\infty(M).
\]
Defining the corresponding nodal $(p,N)$-widths by replacing
$E_N(g)$ with $V_N$, one has
\[
\omega_p(M,g)
=
\lim_{N\to\infty}\nu_p(M,g,V_N).
\]

Indeed, Beck--Becker-Kahn--Hanin \cite{BBKH} used the finite-vanishing-order property
of Laplace--Beltrami eigenfunctions to verify the no-concentration-of-mass condition for
the nodal sweepout in \eqref{eq:MN-nodal-sweepout}.  This condition was
included as a technical hypothesis in the Marques--Neves formulation of
the $p$-widths; see Sections 3.7 and 4 in \cite{MN}.
Zhou subsequently showed that the no-concentration-of-mass
condition can be omitted from the definition of the volume spectrum
without changing the resulting $p$-widths;
see Section 5 in \cite{ZhouDiscretization}. Finally, the proof of the reverse inequality in
Theorem \ref{thm:filtration} uses only density in the $C^\infty$-topology.
\end{remark}

\medskip
\noindent\textbf{Outline.} In Section \ref{sec:discussion}, we state a series of questions on nodal geometry and min-max theory. In Section \ref{sec:counterexample}, we prove Theorem \ref{prop:counterexample-asymptotic-optimality} and Proposition \ref{prop:bbkh-equality}. In Section \ref{sec:filtration}, we prove Theorem \ref{thm:filtration}.

We first explain the main idea of the counterexample to Conjecture \ref{conj:asymptotic-optimality} in dimension two.
For $0<\varepsilon<1$, consider the ellipsoid of revolution
\[
M_\varepsilon
=
\left\{
(x_1,x_2,z)\in\mathbb R^3:
x_1^2+x_2^2+\frac{z^2}{\varepsilon^2}=1
\right\}.
\]
Let $A_\varepsilon=\operatorname{Area}(M_\varepsilon)$ and let
$L_\varepsilon$ denote the length of a meridian.  Separation of
variables gives a sequence of eigenfunctions
\[
u_m(t,\theta)=a_m(t)\cos(m\theta),
\]
where $a_m(t)>0,$ with eigenvalues $\mu_m$ satisfying
\[
\mu_m=m^2+o(m^2).
\]
Moreover, $\{u_m=0\}$ is the union of $m$ meridians, and hence
\[
\cH^1(\{u_m=0\})=mL_\varepsilon.
\]

Let $p_m$ be the spectral index of $\mu_m$.  The classical Weyl law \cite{Weyl}
gives
\[
p_m
=
\left(\frac{A_\varepsilon}{4\pi}+o(1)\right)m^2.
\]
On the other hand, the two-dimensional Weyl law for the volume spectrum of Liokumovich--Marques--Neves \cite{LMN},
whose universal constant was computed by Chodosh--Mantoulidis
\cite{ChodoshMantoulidisWidths}, gives
\[
\omega_p(M_\varepsilon)
\sim
\sqrt{\pi A_\varepsilon p},
\]
and therefore
\[
\omega_{p_m}(M_\varepsilon)
=
\left(\frac{A_\varepsilon}{2}+o(1)\right)m.
\]
Since $u_m\in E_{p_m}(g_\varepsilon)$,
we have
\[
\liminf_{m\to\infty}
\frac{\Phi_{p_m}(M_\varepsilon)}
{\omega_{p_m}(M_\varepsilon)}
\geq
\frac{2L_\varepsilon}{A_\varepsilon}.
\]
As $\varepsilon\downarrow0$, the ellipsoid converges to the unit
disk with multiplicity two, then
\[
A_\varepsilon\longrightarrow2\pi,
\qquad
L_\varepsilon\longrightarrow4.
\]
Consequently,
\[
\frac{2L_\varepsilon}{A_\varepsilon}
\longrightarrow
\frac4\pi>1,
\]
which disproves Conjecture \ref{conj:asymptotic-optimality} for all sufficiently small fixed
$\varepsilon>0$.  The argument in higher dimensions follows the same idea.

We outline the proof of Theorem \ref{thm:filtration}. Let us first explain the main difficulty.  Suppose that a $p$-sweepout
was already represented by a family continuous in the $C^\infty$-topology,
\[
x\longmapsto f_x\in C^\infty(M)\setminus\{0\}
\]
whose nodal sets have finite $n$-dimensional Hausdorff measure.
Then the associated sign boundaries
\[
x\longmapsto \partial\{f_x<0\}
\]
vary continuously in the flat topology.  Moreover, if $P_N$ denotes
the projection onto $E_N(g)$, compactness of the parameter space gives
\[
P_Nf_x\longrightarrow f_x
\]
in $C^\infty(M)$, uniformly in $x$.  Thus it would be natural simply to
replace the family by spectral approximations.

The obstruction is that the Hausdorff measure of the zero set is not
upper semicontinuous under smooth convergence.  This can be
seen on the following example.  On a coordinate cube, writing $y=x_{n+1}$, let
\[
f(y)=
\begin{cases}
e^{-1/y^2},&y\neq0,\\
0,&y=0,
\end{cases}
\qquad
f_k(y)=f(y)+e^{-k}\sin(ky).
\]
Then $f_k\to f$ in $C^\infty$. However, near $y=0$ the zero set of $f_k$
contains approximately $\sqrt k$ parallel hypersurfaces, while
$\{f=0\}=\{y=0\}$.  Consequently, smooth approximation alone gives no
control on nodal measure.

Note that the function $f$ in the example above has infinite vanishing order
at $y=0$.  If instead a fixed smooth function has finite vanishing order at every
point of its nodal set, compactness gives a uniform quantitative version, i.e.
there exist $q<\infty$ and $c>0$ such that, at every zero, some
derivative of order at most $q$ has norm at least $c$.  For a sufficiently
small smooth perturbation the same quantitative finite-vanishing-order property holds.
It was proved by B\"ar \cite{Bar} (see also Beck--Becker-Kahn--Hanin \cite{BBKH}) that this property then implies a uniform local estimate for
the $n$-dimensional Hausdorff measure of the perturbed zero set.
Thus the instability above is caused by the possibility of zeros with
arbitrarily large vanishing order. Our construction below is designed
to exclude this phenomenon uniformly over the whole family.

We start from an almost optimal flat-continuous $p$-sweepout and use
Dey's almost-smooth discretization \cite{Dey}.  On the vertices $v$ of
a sufficiently fine subdivision of the double cover of the cubical complex this
produces Caccioppoli sets $E_v$ and hypersurfaces $\Gamma_v:=E_v\cap E_{Tv},$
such that
\[
\cH^{n}(\Gamma_v)\leq \omega_p(M,g)+o(1).
\]
The main feature of the almost-smooth approximation construction is its locality. For every
simplex $\tau$ there are regions $D_\tau\Subset U_\tau$
such that all the sets $E_v$, $v\prec\tau$, agree outside $D_\tau$.
Moreover, the hypersurfaces are smooth outside $D_\tau$, and their
$\mathcal H^n$-measure in $U_\tau$ can be made arbitrarily small.

It remains to interpolate between the vertex sets.  Put
\[
\sigma_v:=1-2\chi_{E_v}.
\]
After a barycentric subdivision, color the vertices so that the colors
$c(v)\in\{0,\ldots,p\}$ are distinct on every simplex.  If
$\tau=[v_0,\ldots,v_j]$ and $\lambda_i$ are its barycentric
coordinates, define
\[
S_\xi(y)
=
\sum_{i=0}^j
\lambda_i(\xi)e^{c(v_i)h(y)}\sigma_{v_i}(y),
\]
where $h$ is a fixed Morse function coming from the almost-smooth approximation construction, and
set
\[
\vartheta_\xi:=\operatorname{sgn}(S_\xi).
\]

We now smooth the sign family by the heat flow,
\[
F_t(\xi):=H_t\vartheta_\xi.
\]
For every $t>0$ this is a continuous odd family of smooth nonzero
functions.

Outside $U_\tau$, all the sets attached to the vertices of $\tau$ agree, so
the problem reduces to smoothing the sign function of a single smooth
hypersurface. Thus, at scale $\sqrt t$ the heat flow is $C^1$-close to
\[
\operatorname{erf}\left(\frac{r}{2\sqrt t}\right),
\]
where $r$ is signed distance to the hypersurface.  Its zero set, and
the zero set of every sufficiently small smooth perturbation, is
therefore a small normal graph over the smooth hypersurface.

Inside $U_\tau$ there is no such graphical description, and this is
where the finite-vanishing-order idea discussed above enters.
On each region where the vertex signs are fixed,
$S_\xi$ is a polynomial
in $e^h$ of degree at most $p$.  Thus the new interfaces introduced by
the interpolation are contained in a uniformly bounded number of level
sets of the fixed Morse function $h$.  Together with the hypersurfaces attached to the vertices of $\tau$,
this gives a uniform local boundary-density bound for
$\{S_\xi<0\}$.

We now look at the heat flow at scale $\sqrt t$.  The
boundary-density bound makes the rescaled sign functions precompact in
$\operatorname{BV}_{\mathrm{loc}}$, and their heat evolutions converge to a compact
family of nonzero real-analytic Gaussian convolutions.  Compactness and
analyticity then give precisely the desired uniform quantitative
finite-vanishing-order property.  The same property holds for all sufficiently small
smooth perturbations.  As explained above, it prevents the creation of
arbitrarily many nearby nodal sheets and yields a uniform local bound
for their $n$-dimensional measure.  Since Dey's construction makes
the total hypersurface measure inside $U_\tau$ arbitrarily small, the
total nodal measure contribution of this region is also arbitrarily small.
Combining this with the graphical estimate outside $U_\tau$, we obtain,
uniformly in the parameter,
\[
\cH^{n}\bigl(\{\psi=0\}\bigr)
\leq
\omega_p(M,g)+o(1)
\]
whenever $\psi$ is sufficiently $C^\infty$-close to $F_t(\xi)$.

\medskip
\noindent\textbf{Acknowledgements.} The author is grateful to Yevgeny Liokumovich for his invaluable supervision and encouragement throughout this work. The author is thankful to Iosif Polterovich for helpful discussions and his interest in this work. This work was partially supported by Ontario Graduate Scholarship (OGS), University Doctoral Excellence Award (UDEA), and Dr. Sergiy and Tetyana Kryvoruchko Graduate Scholarship in Mathematics.

\medskip
\noindent\textbf{Statement of AI use.}
The proofs of Theorem~\ref{prop:counterexample-asymptotic-optimality} and Proposition~\ref{prop:bbkh-equality} were found with the assistance of ChatGPT 6 Astra Pro, and subsequently verified and rewritten by the author. The problem formulation, the main ideas, and the proof strategy in the proof of Theorem~\ref{thm:filtration} are the author's own, while ChatGPT 5.6 Sol Pro and ChatGPT 6 Astra Pro were used to assist in rigorously working out some technical details. ChatGPT 6 Astra Pro was also used to identify errors in the manuscript. The text of the paper is human-written. The author takes full responsibility for the mathematical content of the manuscript.

\section{Discussion and open questions}
\label{sec:discussion}

Although Theorem~\ref{prop:counterexample-asymptotic-optimality}
disproves Conjecture~\ref{conj:asymptotic-optimality}, a weaker
question raised by Neves \cite{NevesICM} remains open.

\begin{conjecture}[Neves' bounded ratio conjecture, \cite{NevesICM}]
\label{conj:bounded-nodal-ratio}
There exists a constant $C(M,g)>0$ such that
\[
\frac{\Phi_p(M,g)}{\omega_p(M,g)}
\leq
C(M,g)
\]
for every $p\geq1$.
\end{conjecture}

For real-analytic metrics this follows from the work of
Jerison--Lebeau \cite{JerisonLebeau}, the classical
Weyl law \cite{Weyl} and the Weyl law for the volume spectrum of Liokumovich--Marques--Neves \cite{LMN}.
Note that Conjecture~\ref{conj:bounded-nodal-ratio}
would imply the upper bound
in Yau's conjecture for Laplace eigenfunctions \cite{Yau}. The lower bound in Yau's conjecture for Laplace eigenfunctions was famously
proved by Logunov \cite{LogunovLower}.

The two Weyl laws suggest the following quantitative version of the filtration formula.
\begin{conjecture}
For every $\varepsilon>0$,
there exists $C_\varepsilon=C_\varepsilon(M,g)$ such that
\[
\nu_p^{\lceil C_\varepsilon p\rceil}(M,g)
\leq
(1+\varepsilon)\omega_p(M,g)
\]
for every $p\geq1$.
\end{conjecture}

In another direction, one may ask whether the nonlinear
topological families in Theorem~\ref{thm:filtration} can be replaced
by projectivizations of linear subspaces. For $N\geq p$,
define
\[
\ell_p^N(M,g)
:=
\inf_{\substack{V\subset E_N(g)\\ \dim V=p+1}}
\ \sup_{0\neq f\in V}
\cH_g^n(\{f=0\}).
\]
The projectivization of a $(p+1)$-dimensional subspace gives an
admissible family in $\mathcal N_p^N(g)$, and therefore
\[
\nu_p^N(M,g)\leq\ell_p^N(M,g).
\]
Moreover, Proposition~\ref{prop:bbkh-equality} gives
$\ell_p^p(M,g)=\Phi_p(M,g)$.
\begin{conjecture}
For every $p\geq1$,
\[
\omega_p(M,g)
=
\lim_{N\to\infty}\ell_p^N(M,g).
\]
\end{conjecture}

It is also natural to ask for analogous filtration formulas
in other geometric min-max theories.  For example, the
phase-transition spectrum of Gaspar--Guaraco
\cite{GasparGuaraco,GasparGuaracoWeyl}, the half-volume spectrum
of Mazurowski--Zhou
\cite{MazurowskiZhou}, the Simon--Smith min-max theory \cite{Smith} (see also Colding--De Lellis \cite{ColdingDeLellis}), the min-max theory for constant mean
curvature hypersurfaces of Zhou--Zhu \cite{ZhouZhu}, and the free boundary min-max
theory of Li--Zhou \cite{LiZhou}.

It is natural to ask whether the relation between Laplace eigenfunctions
and minimal hypersurfaces can hold at the level of individual
hypersurfaces, rather than only at the level of sweepouts.

\begin{conjecture}
\label{conj:nodal-realization}
Let $(M,g)$ be a closed Riemannian manifold.  Then there exist
$N\geq 0$, a nonzero function $f\in E_N(g),$
and a smooth closed embedded minimal hypersurface $\Sigma\subset M$
such that
\[
\Sigma=\{f=0\}.
\]
\end{conjecture}

In analogy with Yau's abundance conjecture for minimal hypersurfaces \cite{Yau},
one may ask for a much stronger statement.

\begin{conjecture}
\label{conj:nodal-abundance}
Let $(M,g)$ be a closed Riemannian manifold.  Then there exist
infinitely many pairwise distinct smooth closed embedded minimal
hypersurfaces $\Sigma_j\subset M$ such that, for every $j$, there are
$N_j\geq0$ and a nonzero function
$f_j\in E_{N_j}(g)$
satisfying
\[
\Sigma_j=\{f_j=0\}.
\]
\end{conjecture}

Finally, Theorem~\ref{thm:filtration} expresses each element of the
volume spectrum as a limit of finite-dimensional nodal min-max
problems, and may provide a new approach to explicit computations.
There are very few cases in which the complete volume spectrum is
known.  Chodosh--Mantoulidis \cite{ChodoshMantoulidisWidths} computed
the $p$-widths of the round two-sphere, and Marx-Kuo computed $p$-widths
of the standard real projective plane \cite{MarxKuoRP2}, Zoll metrics on two-spheres \cite{MarxKuoZoll} and round hemispheres \cite{MarxKuoHemisphere}.
In higher
dimensions, round spheres and flat tori seem to be natural first
targets, since their Laplace eigenspaces are explicitly understood.
Determining the universal constant
$a_n$ in the Weyl law for the volume spectrum remains open for
$n\geq2$; the case $n=1$ was computed by
Chodosh--Mantoulidis \cite{ChodoshMantoulidisWidths}.

\section{A counterexample to the asymptotic optimality conjecture}
\label{sec:counterexample}
\begin{lemma}
Suppose that
$f\in C^\infty(M)$ and that $Z=f^{-1}(0)$ is countably
$n$-rectifiable with finite measure.  If $s_j\downarrow0$ and both
$s_j$ and $-s_j$ are regular values of $f$, then
\begin{equation}
\label{eq:two-sided-level-measure}
\cH_g^{n}(Z)
\leq
\liminf_{j\to\infty}
\frac{
\cH_g^{n}(f^{-1}(s_j))
+
\cH_g^{n}(f^{-1}(-s_j))
}{2}.
\end{equation}
\end{lemma}
\begin{proof}
Let $L_1$ be the length of the shortest closed geodesic of $(M,g)$ and
fix $0<T<L_1$. Let
\[
SM:=\{(x,v)\in TM:|v|_g=1\}
\]
be the unit tangent bundle of $M$, and let
\[
dL(x,v):=dV_g(x)\,d\sigma_x(v)
\]
denote the Liouville measure on $SM$, where $d\sigma_x$ is the standard
spherical measure on the unit sphere $S_xM\subset T_xM$. For a countably $n$-rectifiable set $S\subset M$,
set
\[
N_S(x,v)
:=
\#\bigl\{
t\in[-T,T]:
\gamma_{x,v}(t)\in S
\bigr\},
\qquad
\gamma_{x,v}(t)=\exp_x(tv).
\]
By the Crofton formula (see e.g. Proposition 9 in Zelditch \cite{ZelditchCrofton}) if $S$ is
a smooth hypersurface, then
\begin{equation}
\label{eq:crofton}
\cH_g^{n}(S)
=
\frac{1}{c_n T}\int_{SM}N_S(x,v)\,dL(x,v),
\end{equation}
where $c_n>0$ depends only on the dimension.  By the area formula, the same identity holds for countably
$n$-rectifiable sets of finite measure.

Applying \eqref{eq:crofton} to $Z$, we see that $N_Z(x,v)<\infty$ for
almost every $(x,v)\in SM$.  After removing another null set, we may
also assume that the endpoints of $\gamma_{x,v}$ do not belong to $Z$.
Fix such a geodesic segment.  Around each of its finitely many
intersections with $Z$, choose a small closed interval containing no
other intersection with $Z$, with all these intervals pairwise
disjoint.  For all sufficiently large $j$, $s_j$ is smaller than the
absolute value of $f\circ\gamma_{x,v}$ at every endpoint of these
intervals.

Consider one such interval.  If the endpoint values of
$f\circ\gamma_{x,v}$ have opposite signs, the intermediate value
theorem gives one intersection with $f^{-1}(s_j)$ and one with
$f^{-1}(-s_j)$.  If both endpoint values are positive, it gives two
intersections with $f^{-1}(s_j)$, one on each side of the zero. If
both are negative, it similarly gives two intersections with
$f^{-1}(-s_j)$.  Summing over the intervals yields
\[
N_Z(x,v)
\leq
\frac{1}{2}
\liminf_{j\to\infty}
\left(
N_{f^{-1}(s_j)}(x,v)
+
N_{f^{-1}(-s_j)}(x,v)
\right)
\]
for almost every $(x,v)\in SM$.

Integrating this inequality and using Fatou's lemma, we obtain
\begin{align*}
\cH_g^{n}(Z)
&=
\frac{1}{c_nT}\int_{SM}N_Z\,dL\\
&\leq
\frac{1}{2c_nT}\int_{SM}
\liminf_{j\to\infty}
\left(
N_{f^{-1}(s_j)}
+
N_{f^{-1}(-s_j)}
\right)dL\\
&\leq
\frac{1}{2c_nT}\liminf_{j\to\infty}
\int_{SM}
\left(
N_{f^{-1}(s_j)}
+
N_{f^{-1}(-s_j)}
\right)dL\\
&=
\frac{1}{2}
\liminf_{j\to\infty}
\left(
\cH_g^{n}(f^{-1}(s_j))
+
\cH_g^{n}(f^{-1}(-s_j))
\right).
\end{align*}
Since $\pm s_j$ are regular values and hence the
two level sets are smooth hypersurfaces, the last equality follows from
\eqref{eq:crofton}. This completes the proof of the lemma.
\end{proof}

\begin{proof}[Proof of Proposition \ref{prop:bbkh-equality}]
Fix $0\neq f\in E_p(g)$.  Nonzero finite linear combinations of Laplace eigenfunctions have
finite vanishing order by Lin \cite{Lin91}. Lemma 3 in B\"ar  \cite{Bar}
implies that their zero sets are countably
$n$-rectifiable and have locally finite $n$-dimensional
Hausdorff measure. By Sard's theorem, choose
$s_j\downarrow0$ such that both $\pm s_j$ are regular values.
Then, since $f\pm s_j\in E_p(g)$, we have
\[
\cH_g^{n}(f^{-1}(\pm s_j))
=
\Mmass_g\bigl(\partial\{f<\pm s_j\}\bigr)
\leq \Phi_p(M,g).
\]
By \eqref{eq:two-sided-level-measure},
\[
\cH_g^{n}(\{f=0\})\leq\Phi_p(M,g).
\]
Taking the supremum over $f$ gives
\[
\sup_{0\neq f\in E_p(g)}
\cH_g^{n}(\{f=0\})
\leq \Phi_p(M,g).
\]
The reverse inequality follows from $\partial\{f<0\}\subset\{f=0\}$.
This proves the first equality in Proposition \ref{prop:bbkh-equality}.

Finally, identify $\mathbb P(E_p(g))$ with $\mathbb{RP}^p$.
If $\Psi:X\longrightarrow\mathbb{RP}^p$ is such that
$\Psi^*(a_p^p)\neq0,$ then $\Psi$ is surjective.
Indeed, if $\Psi$ doesn't contain a point $q\in\mathbb{RP}^p$ in the image, then it factors through $\mathbb{RP}^p\setminus\{q\}$, which deformation retracts onto $\mathbb{RP}^{p-1}$. Hence $a_p^p$ restricts to zero on $\mathbb{RP}^p\setminus\{q\}$, contradicting $\Psi^*(a_p^p)\neq0$. Thus every family in $\mathcal N_p^p(g)$ has
maximal nodal measure
\[
\sup_{0\neq f\in E_p(g)}
\cH_g^{n}(\{f=0\}).
\]
Conversely, the identity map of $\mathbb P(E_p(g))$ belongs to $\mathcal N_p^p(g)$.  This proves the
second equality.
\end{proof}

\begin{proof}[Proof of Theorem \ref{prop:counterexample-asymptotic-optimality}]
Denote $\sigma_j:=\mathcal H^j(S^j).$ Combining the Weyl law for the Laplace spectrum \cite{Weyl} with the Liokumovich--Marques--Neves' Weyl law for the volume spectrum \cite{LMN}, there is
$\beta_n>0$ such that, on every closed $(n+1)$-manifold,
\begin{equation}
\label{eq:beta-d}
\frac{\omega_p(M,g)}
{\operatorname{Vol}(M,g)\sqrt{\lambda_p(M,g)}}
\longrightarrow \beta_n.
\end{equation}

Considering the sweepouts given by linear combinations of the first $p+1$ Laplace eigenfunctions on $(M^{n+1},g) = (S^{n+1},g_{\operatorname{round}})$, and using Theorem 2 in Beck--Becker-Kahn--Hanin \cite{BBKH} and Theorem 3 in Gichev \cite{Gichev}, we obtain
\begin{equation}
\label{eq:beta-upper}
\beta_n\leq \frac{\sigma_{n}}{\sigma_{n+1}}.
\end{equation}

For $0<\varepsilon<1$, consider the ellipsoid
\[
M_\varepsilon
=
\left\{
(x,z)\in\mathbb R^{n+1}\times\mathbb R:
|x|^2+\frac{z^2}{\varepsilon^2}=1
\right\}
\]
with the induced metric, and the following parametrization
\[
(t,\xi)\longmapsto
(\sin t\,\xi,\varepsilon\cos t),
\qquad
0\leq t\leq\pi,\quad \xi\in S^{n}.
\]
Then
\[
g_\varepsilon
=
q_\varepsilon(t)^2dt^2
+\sin^2t\,g_{S^{n}},
\]
where
\[
q_\varepsilon(t)
=
\sqrt{\cos^2t+\varepsilon^2\sin^2t}.
\]
Put
\[
I_j(\varepsilon)
:=
\int_0^\pi
q_\varepsilon(t)\sin^j t\,dt.
\]
Then
\begin{equation}
\label{eq:ellipsoid-volume}
\operatorname{Vol}(M_\varepsilon)
=
\sigma_{n}I_{n}(\varepsilon).
\end{equation}

For $m\geq1$, let
\[
Y_m(\xi)
=
\operatorname{Re}(\xi_1+i\xi_2)^m.
\]
Thus
\[
-\Delta_{S^{n}}Y_m
=
\kappa_mY_m,
\]
where
\[
\kappa_m=m(m+n-1).
\]
The subspace of functions of the form $u(t,\xi)=a(t)Y_m(\xi)$
is invariant under $\Delta_{g_\varepsilon}$.  Let $\mu_m$ be the
lowest eigenvalue of $-\Delta_{g_\varepsilon}$ in this subspace and
denote a corresponding eigenfunction as
\[
u_m(t,\xi)=a_m(t)Y_m(\xi).
\]
Computing the Rayleigh quotient, we have
\[
\mu_m=\inf_a Q_m(a),
\]
where
\[
Q_m(a)
=
\frac{
\displaystyle
\int_0^\pi
\left(
\frac{|a'|^2}{q_\varepsilon^2}
+
\frac{\kappa_m}{\sin^2t}a^2
\right)
q_\varepsilon\sin^{n}t\,dt
}{
\displaystyle
\int_0^\pi
a^2q_\varepsilon\sin^{n}t\,dt
}.
\]
Since $\sin t\leq1$, we have $\mu_m\geq\kappa_m.$
Conversely, given $\eta>0$, choose a fixed nonzero smooth function
$a$ supported sufficiently close to $t=\pi/2$ so that
$\sin^{-2}t\leq1+\eta$ on its support.  Then
\[
\mu_m
\leq
(1+\eta)\kappa_m+C_{\varepsilon,\eta}.
\]
Hence
\begin{equation}
\label{eq:mu-m-asymptotic}
\frac{\mu_m}{m^2}\longrightarrow1.
\end{equation}

Since $Q_m(|a_m|)=Q_m(a_m),$ then $a_m$ can be chosen nonnegative.
By the uniqueness of the solutions of the Euler--Lagrange equation, we have
$a_m(t)>0$ for $0<t<\pi$.
Consequently, $\{u_m=0\}$ is the union of the $m$ meridional
hypersurfaces corresponding to the $m$ great $(n-1)$-spheres in
$\{Y_m=0\}\subset S^n$. Hence
\begin{equation}
\label{eq:um-nodal-mass}
\cH_{g_\varepsilon}^{n}(\{u_m=0\})
=
m\sigma_{n-1}I_{n-1}(\varepsilon).
\end{equation}

Let
\[
p_m
=
\#\{j:\lambda_j(M_\varepsilon)\leq\mu_m\}-1,
\]
where eigenvalues are counted with multiplicity.  Then $\lambda_{p_m}(M_\varepsilon)=\mu_m$
and $E_{p_m}(g_\varepsilon)$ contains $u_m$. By Proposition \ref{prop:bbkh-equality} and \eqref{eq:um-nodal-mass},
\[
\Phi_{p_m}(M_\varepsilon)
\geq
m\sigma_{n-1}I_{n-1}(\varepsilon).
\]
Since $p_m\to\infty$, \eqref{eq:beta-d},
\eqref{eq:ellipsoid-volume}, and \eqref{eq:mu-m-asymptotic} imply
\[
\omega_{p_m}(M_\varepsilon)
=
\bigl(\beta_n+o(1)\bigr)
\sigma_{n}I_{n}(\varepsilon)m.
\]
It follows that
\begin{equation}
\label{eq:ratio-ellipsoid}
\liminf_{m\to\infty}
\frac{\Phi_{p_m}(M_\varepsilon)}
{\omega_{p_m}(M_\varepsilon)}
\geq
\frac{
\sigma_{n-1}I_{n-1}(\varepsilon)
}{
\beta_n\sigma_{n}I_{n}(\varepsilon)
}
\geq
\frac{\sigma_{n+1}\sigma_{n-1}}{\sigma_{n}^2}
\frac{I_{n-1}(\varepsilon)}
{I_{n}(\varepsilon)},
\end{equation}
where the last inequality follows from \eqref{eq:beta-upper}.

As $\varepsilon\downarrow0$,
\[
q_\varepsilon(t)\longrightarrow |\cos t|,
\]
and hence
\[
I_j(\varepsilon)
\longrightarrow
\int_0^\pi
|\cos t|\sin^jt\,dt
=
\frac{2}{j+1}.
\]
Thus the right-hand side of \eqref{eq:ratio-ellipsoid} converges to
\begin{equation}
\label{eq:dimensional-gap}
C_n
:=
\frac{n+1}{n}
\frac{\sigma_{n+1}\sigma_{n-1}}{\sigma_{n}^2}.
\end{equation}
We claim that $C_n>1$. Since
\[
\sigma_j
=
\frac{2\pi^{(j+1)/2}}
{\Gamma((j+1)/2)},
\]
we have
\[
\frac{\sigma_{n+1}\sigma_{n-1}}{\sigma_{n}^2}
=
\frac{
\Gamma((n+1)/2)^2
}{
\Gamma((n+2)/2)\Gamma(n/2)
}.
\]
Hence
\[
C_n
=
\frac{n+1}{n}
\frac{
\Gamma((n+1)/2)^2
}{
\Gamma((n+2)/2)\Gamma(n/2)
}
=
\frac{2(n+1)}{n^2}
\frac{
\Gamma((n+1)/2)^2
}{
\Gamma(n/2)^2
},
\]
where in the last equality we used
\[
\Gamma\left(\frac{n+2}{2}\right)
=
\frac{n}{2}
\Gamma\left(\frac{n}{2}\right).
\]

By the strict log-convexity of the Gamma function,
\[
\Gamma\left(\frac{n+2}{2}\right)^2
<
\Gamma\left(\frac{n+1}{2}\right)
\Gamma\left(\frac{n+3}{2}\right).
\]
Using
\[
\Gamma\left(\frac{n+3}{2}\right)
=
\frac{n+1}{2}\Gamma\left(\frac{n+1}{2}\right)
\]
and again
\[
\Gamma\left(\frac{n+2}{2}\right)
=
\frac{n}{2}
\Gamma\left(\frac{n}{2}\right),
\]
we obtain
\[
\frac{n^2}{4}
\Gamma\left(\frac{n}{2}\right)^2
<
\frac{n+1}{2}
\Gamma\left(\frac{n+1}{2}\right)^2.
\]
Equivalently,
\[
C_n
=
\frac{2(n+1)}{n^2}
\frac{
\Gamma((n+1)/2)^2
}{
\Gamma(n/2)^2
}
>1.
\]

Thus, by \eqref{eq:ratio-ellipsoid}, taking $\varepsilon>0$ sufficiently small gives
\[
\limsup_{p\to\infty}
\frac{\Phi_p(M_\varepsilon)}
{\omega_p(M_\varepsilon)}
\geq
\liminf_{m\to\infty}
\frac{\Phi_{p_m}(M_\varepsilon)}
{\omega_{p_m}(M_\varepsilon)}
>1.
\]
This completes the proof.
\end{proof}

\section{Filtration formula}
\label{sec:filtration}

The following proposition follows from Dey's almost-smooth discrete sweepout construction (see Section 3 in \cite{Dey}).

\begin{proposition}[Dey's almost-smooth discretization, \cite{Dey}]
\label{prop:dey-almost-smooth-discretization}
Let $X$ be a cubical complex with $\dim X=p$, and let
$\Phi:X\longrightarrow\cZ_{n}(M;\Ztwo)$
be a $p$-sweepout.  Set
\[
L:=\sup_{x\in X}\Mmass_g(\Phi(x)),
\]
and suppose that
$L<\infty.$ Let $\pi:\widetilde X\longrightarrow X$
be the corresponding double cover with deck involution $T$. 

For every $\delta>0$, there is a $T$-equivariant
barycentric triangulation of $\widetilde X$ and a map
\[
v\longmapsto E_v\in\cC(M)
\]
defined on its vertices, where each $E_v$ is chosen as a closed representative. Moreover, Moreover, there exist a smooth Morse function $h:M\longrightarrow[0,1]$
and an open set $W_0\subset M$ such that the following properties hold.

\begin{itemize}
\item[(i)]
For every vertex $v$,
\[
E_v\cup E_{Tv}=M.
\]
Set
\[
\Gamma_v:=E_v\cap E_{Tv}.
\]
Each $\Gamma_v$ is a finite union of smooth hypersurface pieces.  Every
nonempty intersection of $r$ distinct
pieces is transverse and has codimension $r$.  All intersections of at least
two pieces are contained in $W_0$.  In particular, every $\Gamma_v$ is a
smooth hypersurface on $M\setminus W_0$.

\item[(ii)]
For every vertex $v$,
\begin{equation}
\label{eq:almost-smooth-global-area}
\cH^{n}(\Gamma_v)
\leq
L+C_p\delta.
\end{equation}

\item[(iii)]
For every simplex $\tau$ there are open sets
\[
W_0\subset D_\tau\Subset U_\tau
\]
such that $U_\tau$ has smooth boundary, $dh\neq0$ on
$\partial U_\tau$, $h|_{\partial U_\tau}$ is Morse, and
$\partial U_\tau$ is transverse to every smooth hypersurface piece
occurring in $\Gamma_v$, $v\prec\tau$. Moreover, for every pair of vertices $v,w\prec\tau$,
\begin{equation}
\label{eq:almost-smooth-simplex-locality}
E_v\cap(M\setminus D_\tau)
=
E_w\cap(M\setminus D_\tau),
\end{equation}
and
\begin{equation}
\label{eq:almost-smooth-local-area}
\sum_{v\prec\tau}
\cH^{n}(\Gamma_v\cap U_\tau)
+
\sup_{s\in[0,1]}
\cH^{n}\bigl(h^{-1}(s)\cap U_\tau\bigr)
\leq
C_p\delta.
\end{equation}

\item[(iv)]
The vertices admit a $T$-invariant coloring
\begin{equation}
\label{eq:barycentric-coloring}
c:\widetilde X_0\longrightarrow\{0,1,\ldots,p\}
\end{equation}
which is injective on the vertex set of every simplex.
\end{itemize}
\end{proposition}

\begin{proof}
Properties (i)--(iii) follow directly from Dey's almost-smooth discrete sweepout construction (see Section 3 in \cite{Dey}).

It remains to prove (iv). Take the barycentric subdivision of
the simplicial complex.  A new vertex is the barycenter
$b_\sigma$ of a simplex $\sigma$ of the preceding subdivision.  Choose a
vertex
\[
a(\sigma)\prec\sigma
\]
for every $\sigma$, equivariantly under the deck involution:
\[
a(T\sigma)=T a(\sigma).
\]
Set
\[
E_{b_\sigma}:=E_{a(\sigma)}.
\]

A simplex of the barycentric subdivision has the form
\[
\tau=
[b_{\sigma_0},\ldots,b_{\sigma_j}],
\]
where
\[
\sigma_0\prec\sigma_1\prec\ldots\prec\sigma_j.
\]
For every $i$, the chosen vertex $a(\sigma_i)$ is a vertex of the largest
simplex $\sigma_j$.  Hence all the sets attached to the vertices of
$\tau$ come from vertices of the single preceding simplex $\sigma_j$.
Therefore, (i)--(iii) still hold on the barycentric subdivision.

Finally define
\[
c(b_\sigma):=\dim\sigma.
\]
If
\[
[b_{\sigma_0},\ldots,b_{\sigma_j}]
\]
is a barycentric simplex, then
\[
\dim\sigma_0
<
\dim\sigma_1
<
\ldots
<
\dim\sigma_j.
\]
Thus its vertices have distinct colors.  Since $\dim X\leq p$, all colors
belong to $\{0,\ldots,p\}$, and since $T$ preserves dimensions, the
coloring is $T$-invariant.
\end{proof}

\begin{proof}[Proof of Theorem \ref{thm:filtration}]
By Theorem 2 in Beck--Becker-Kahn--Hanin \cite{BBKH},
\[
\omega_p(M,g)\leq \nu_p^N(M,g)
\]
for every $N\geq p$, so it remains to prove the reverse inequality in the
limit. Fix $\eta>0$.  Choose a finite cubical complex $X$, with
$\dim X\leq p$, and a $p$-sweepout
$\Phi:X\longrightarrow \mathcal B(M;\mathbb Z_2).$
such that
\[
L:=\sup_{x\in X}\Mmass_g(\Phi(x))
\leq
\omega_p(M,g)+\eta.
\]
Let $\pi:\widetilde X\longrightarrow X$
be the corresponding double cover, with deck involution $T$, and let
$\alpha:=\Phi^*(\bar\lambda)\in H^1(X;\Ztwo).$
Then $\alpha^p\neq0.$

The case $n=0$ is trivial.  Assume $n\geq1$.  We divide the
construction into four steps.

\medskip
\noindent
\textbf{Step 1: Dey's almost-smooth discretization.}
Fix $\delta>0$, to be chosen at the end. Applying
Proposition~\ref{prop:dey-almost-smooth-discretization} to $\Phi$, we obtain a
map
\[
v\longmapsto E_v
\]
on the vertices of $\widetilde X$.  Write
\[
\Gamma_v:=E_v\cap E_{Tv}.
\]
We also obtain the function
\[
h:M\longrightarrow[0,1],
\]
the neighborhood $W_0$, the regions
\[
W_0\subset D_\tau\Subset U_\tau
\]
associated with each simplex $\tau$, and the $T$-invariant coloring
\[
c:\widetilde X_0\longrightarrow\{0,\ldots,p\},
\]
such that properties (i)--(iv) hold.

\medskip
\noindent
\textbf{Step 2: interpolation and heat smoothing.}
For a vertex $v$, define
\[
\sigma_v:=1-2\chi_{E_v}\in L^\infty(M;\{-1,1\}).
\]
Then $\sigma_{Tv}=-\sigma_v$ a.e. on $M$.
Let $\tau=[v_0,\ldots,v_j]$ be a simplex and let
$\lambda_0,\ldots,\lambda_j$ be its barycentric coordinates. For $\xi\in\tau$, define
\begin{equation}
\label{eq:piecewise-selector}
S_\xi(y)
:=
\sum_{i=0}^j
\lambda_i(\xi)e^{c(v_i)h(y)}\sigma_{v_i}(y).
\end{equation}
Formula \eqref{eq:piecewise-selector} agrees on common faces and therefore
defines a continuous odd map
\[
S:\widetilde X\longrightarrow L^2(M).
\]
Set
\[
a_\xi(y)
:=
\sum_{i=0}^j
\lambda_i(\xi)e^{c(v_i)h(y)}.
\]
Then
\begin{equation}
\label{eq:weight-bounds}
1\leq a_\xi\leq e^p.
\end{equation}
By \eqref{eq:almost-smooth-simplex-locality}, for every vertex
$v\prec\tau$,
\begin{equation}
\label{eq:selector-locality}
S_\xi=a_\xi\sigma_v
\end{equation}
a.e. on $M\setminus D_\tau$.

For $\xi\in\widetilde X$, set
\[
A_\xi:=\operatorname{int}\{S_\xi<0\}
\]
and
\[
\vartheta_\xi:=1-2\chi_{A_\xi}.
\]
Thus $\vartheta_\xi\in L^\infty(M;\{-1,1\})$.  We first note that
\begin{equation}
\label{eq:selector-zero-volume}
\cH^{n+1}(\{S_\xi=0\})=0
\end{equation}
for every $\xi\in\widetilde X$.  Indeed, suppose that
$\xi\in\tau=[v_0,\ldots,v_j]$ and set
\[
\mathscr G_\tau:=\bigcup_{i=0}^j\Gamma_{v_i}.
\]
Let $C$ be a connected component of $M\setminus\mathscr G_\tau$. For every $i$, the function $\sigma_{v_i}$ is constant on $C$. Denote
\[
\sigma_{v_i}|_C=\epsilon_i,
\]
where $\epsilon_i\in\{-1,1\}$.
Then on $C$,
\[
S_\xi=P(e^h),
\]
where
\[
P(Y)
:=
\sum_{i=0}^j
\lambda_i(\xi)\epsilon_iY^{c(v_i)}.
\]
The colors $c(v_i)$ are distinct, and at least one barycentric coordinate
$\lambda_i(\xi)$ is nonzero.  Thus $P$ is a nonzero polynomial of degree
at most $p$. There are only finitely many sign vectors
$\epsilon\in\{-1,1\}^{j+1}$. Hence, outside $\mathscr G_\tau$,
$\{S_\xi=0\}$ is contained in a finite union of level sets of $h$.  This proves
\eqref{eq:selector-zero-volume}.

The map
\[
\vartheta:\widetilde X\longrightarrow L^2(M),
\qquad
\vartheta(\xi):=\vartheta_\xi,
\]
is continuous. Indeed, suppose that $\xi_k\to\xi$.  For every
$a>0$,
\[
\{\vartheta_{\xi_k}\neq\vartheta_\xi\}
\subset
\{|S_\xi|\leq a\}
\cup
\{|S_{\xi_k}-S_\xi|\geq a\}.
\]
By Chebyshev's inequality,
\[
\|\vartheta_{\xi_k}-\vartheta_\xi\|_{L^2}^2
\leq
4\cH^{n+1}(\{|S_\xi|\leq a\})
+
4a^{-2}\|S_{\xi_k}-S_\xi\|_{L^2}^2.
\]
First let $k\to\infty$, and then let $a\downarrow0$.  Since
$S:\widetilde X\to L^2(M)$ is continuous,
\eqref{eq:selector-zero-volume} implies
$\vartheta_{\xi_k}\longrightarrow\vartheta_\xi$
in $L^2(M)$. Moreover, since $S_{T\xi}=-S_\xi$, we have
$\vartheta_{T\xi}=-\vartheta_\xi$
almost everywhere.

Let $H_t$ denote the heat flow on $(M,g)$ (see
Minakshisundaram--Pleijel \cite{MP} and Rosenberg \cite{Rosenberg}).
Thus, for $f\in L^2(M)$, the function
\[
u(t,\cdot):=H_tf
\]
is the solution of
\[
\partial_tu-\Delta_gu=0,
\qquad
u(0,\cdot)=f.
\]
For every $t>0$, the heat flow defines an injective continuous linear map
\[
H_t:L^2(M)\longrightarrow C^\infty(M).
\]
Define
\begin{equation}
\label{eq:smoothed-selector}
F_t(\xi):=H_t\vartheta_\xi.
\end{equation}
Hence, for every $t>0$,
\[
F_t:\widetilde X\longrightarrow C^\infty(M)\setminus\{0\}
\]
is a continuous odd map.

\medskip
\noindent
\textbf{Step 3: the Hausdorff measure estimate.}
We claim that, for every $\kappa>0$, there exists $t>0$ so that
the following holds.  If $\xi_k\to\xi$ in $\widetilde X$ and
\[
\|\psi_k-F_t(\xi_k)\|_{C^m(M)}
\longrightarrow0
\]
for every $m$, then
\begin{equation}
\label{eq:stable-heat-nodal-bound}
\limsup_{k\to\infty}
\cH^{n}(\{\psi_k=0\})
\leq
L+C_p\delta+\kappa.
\end{equation}

After passing to a subsequence, we may suppose that $\xi_k$ and $\xi$
belong to one simplex $\tau=[v_0,\ldots,v_j].$
Fix a vertex $v\prec\tau$.

We first record some standard results about the heat flow.

\begin{itemize}

\item If $A,B\subset M$ are separated by
\[
\operatorname{dist}_g(A,B)\geq\rho>0
\]
and $f\in L^\infty(M)$ is supported in $B$, then, for every $m\geq0$ and $0<t\leq1$,
\begin{equation}
\label{eq:heat-off-diagonal-general}
\sup_A
|\nabla^mH_tf|
\leq
C_m
t^{-(n+m+2)}
\exp\left(-\frac{\rho^2}{4t}\right).
\|f\|_{L^\infty}.
\end{equation}
This follows from Theorem 3.5 in Ludewig \cite{Ludewig}.

\item For $z\in\mathbb R^{n+1},$ let
\[
G(z):=(4\pi)^{-(n+1)/2}e^{-|z|^2/4}.
\]
Suppose that $t_k\downarrow0$, $x_k\in M$, and
\[
\iota_k:\mathbb R^{n+1}\longrightarrow T_{x_k}M
\]
are orthonormal isometries.  Set
\[
\Phi_k(z):=\exp_{x_k}(\sqrt{t_k}\,\iota_k z)
\]
whenever the right-hand side is defined.  Let $f_k\in L^\infty(M)$
satisfy
\[
\sup_k\|f_k\|_{L^\infty}<\infty,
\]
and suppose that
\[
f_k\circ\Phi_k\longrightarrow f_\infty
\]
in $L^1_{\mathrm{loc}}(\mathbb R^{n+1})$. Then
\begin{equation}
\label{eq:rescaled-heat-convergence}
H_{t_k}f_k\circ\Phi_k
\longrightarrow
G*f_\infty
\end{equation}
in $C^\infty_{\mathrm{loc}}(\mathbb R^{n+1}).$ This follows from Theorems 1.1 and 3.5 in Ludewig \cite{Ludewig}.

\item For $t>0$,
\begin{equation}
\label{eq:heat-semigroup-bounds}
\|H_tf\|_{L^\infty}
\leq
\|f\|_{L^\infty},
\qquad
\|\nabla H_tf\|_{L^\infty}
\leq
Ct^{-1/2}\|f\|_{L^\infty}.
\end{equation}
This follows from Theorem 7.13 in Grigor'yan \cite{Grig}.

\end{itemize}

\medskip
\noindent
\textbf{Step 3 (a): the Hausdorff measure estimate on $M \setminus U_\tau$.}
By \eqref{eq:selector-locality} and \eqref{eq:weight-bounds},
\[
A_{\xi_k}\cap(M\setminus D_\tau)
=
E_v\cap(M\setminus D_\tau)
\]
modulo null sets.  Consequently,
$\vartheta_{\xi_k}=\sigma_v$
almost everywhere on $M\setminus D_\tau$.  Hence
\[
F_t(\xi_k)-H_t\sigma_v
=
H_t(\vartheta_{\xi_k}-\sigma_v),
\]
where $\vartheta_{\xi_k}-\sigma_v$ is supported in
$\overline{D_\tau}$ and has $L^\infty$-norm at most $2$.  Set
\[
d_\tau
:=
\operatorname{dist}_g
\bigl(
\overline{D_\tau},M\setminus U_\tau
\bigr)>0.
\]
By \eqref{eq:heat-off-diagonal-general}, for every $m\geq0$,
\begin{equation}
\label{eq:heat-off-diagonal}
\sup_k
\|F_t(\xi_k)-H_t\sigma_v\|_{C^m(M\setminus U_\tau)}
\leq
C_m t^{-(n+m+2)}
\exp\left(-\frac{d_\tau^2}{4t}\right).
\end{equation}

Since $D_\tau\Subset U_\tau$ and $\Gamma_v$ is smooth on $M\setminus D_\tau$, the compact sets $\Gamma_v\cap(M\setminus U_\tau)$ and $\overline{D_\tau}$ are at a positive distance from each other.  By the tubular
neighborhood theorem, there exists $r_\tau>0$ such that the exponential map
\[
(q,r)\longmapsto \exp_q(r\nu_q),
\qquad
|r|<r_\tau,
\]
is one-to-one for $q \in \Gamma_v\cap(M\setminus U_\tau)$.
Choosing the normal so that $\sigma_v=1$ on the side $r>0$, the signed
distance $r$ is smooth in the tube $\mathcal T'$ and satisfies
\[
\sigma_v=\operatorname{sgn}(r),
\qquad
|\nabla r|_g=1.
\]
Choose a smaller tubular neighborhood $\mathcal T\Subset\mathcal T'.$

\begin{lemma}
There exist $C>0$ and $t_0>0$ such that, for every $0<t<t_0$,
\begin{equation}
\label{eq:heat-sign-model-comparison}
\sup_{\mathcal T}
\left|
H_t\sigma_v-\operatorname{erf}
\left(
\frac{r}{2\sqrt t}
\right)
\right|
\leq C\sqrt t
\end{equation}
and
\begin{equation}
\label{eq:heat-sign-gradient-comparison}
\sup_{\mathcal T}
\left|
\nabla\bigl(H_t\sigma_v-\operatorname{erf}
\left(
\frac{r}{2\sqrt t}
\right)\bigr)
\right|
\leq C.
\end{equation}
\end{lemma}

\begin{proof}
Let
\[
\chi\in C_c^\infty(\mathcal T'),
\qquad
0\leq\chi\leq1,
\]
such that $\chi\equiv1$ on $\overline{\mathcal T}$.  In particular, both
$\operatorname{spt}\nabla\chi$ and $\operatorname{spt}(1-\chi)$ are a fixed positive distance from $\mathcal T$.

Set
\[
\varphi(t,\rho)
:=
\operatorname{erf}
\left(
\frac{\rho}{2\sqrt t}
\right)
\]
and
\[
\alpha(t,x)
:=
\chi(x)\varphi(t,r(x)),
\]
extended by zero outside $\mathcal T'$.  Since
\[
\varphi(t,\rho)\longrightarrow\operatorname{sgn}(\rho)
\]
for $\rho\neq0$, we have
\[
\alpha(t,\cdot)\longrightarrow\chi\sigma_v
\]
in $L^2(M)$ as $t\downarrow0$.

We decompose
\[
H_t\sigma_v-\varphi(t,r)
=
\bigl(H_t(\chi\sigma_v)-\alpha(t)\bigr)
+
H_t((1-\chi)\sigma_v)
\]
on $\mathcal T$.

The function $\varphi$ solves the one-dimensional heat equation
\[
\partial_t\varphi-\partial_\rho^2\varphi=0,
\]
and
\[
\partial_\rho\varphi(t,\rho)
=
\frac{1}{\sqrt{\pi t}}
\exp\left(
-\frac{\rho^2}{4t}
\right).
\]
Using $|\nabla r|_g=1$, we have
\[
E(t,x)
:=
(\partial_t-\Delta_g)\alpha(t,x)
=
-\chi\,\partial_\rho\varphi(t,r)\Delta_g r
-2\partial_\rho\varphi(t,r)
\langle\nabla\chi,\nabla r\rangle_g
-\varphi(t,r)\Delta_g\chi.
\]
Write
\[
E=E_0+E_\chi,
\]
where
\[
E_0
=
-\chi\,\partial_\rho\varphi(t,r)\Delta_g r
\]
and
\[
E_\chi
=
-2\partial_\rho\varphi(t,r)
\langle\nabla\chi,\nabla r\rangle_g
-\varphi(t,r)\Delta_g\chi.
\]

Since $r$ is smooth on the fixed tube $\mathcal T'$,
\[
|\Delta_g r|\leq C,
\]
and hence
\begin{equation}
\label{eq:main-error}
\|E_0(t,\cdot)\|_{L^\infty}
\leq
Ct^{-1/2}.
\end{equation}
Since $\chi$ is fixed,
\[
|\nabla\chi|+|\Delta_g\chi|\leq C,
\]
and therefore
\begin{equation}
\label{eq:cutoff-error}
\|E_\chi(t,\cdot)\|_{L^\infty}
\leq
C(1+t^{-1/2}).
\end{equation}

Consider
\[
w(t)
:=
H_t(\chi\sigma_v)-\alpha(t).
\]
Then
\[
(\partial_t-\Delta_g)w=-E,
\qquad
w(0)=0.
\]
By Ball's variation-of-constants formula \cite{Ball},
\begin{equation}
\label{eq:duhamel}
w(t)
=
-\int_0^tH_{t-s}E(s)\,ds.
\end{equation}

By \eqref{eq:heat-semigroup-bounds} and \eqref{eq:main-error},
\[
\left\|
\int_0^tH_{t-s}E_0(s)\,ds
\right\|_{L^\infty}
\leq
C\int_0^t s^{-1/2}\,ds
\leq
C\sqrt t,
\]
and
\[
\left\|
\nabla\int_0^tH_{t-s}E_0(s)\,ds
\right\|_{L^\infty}
\leq
C\int_0^t
(t-s)^{-1/2}s^{-1/2}\,ds
=
C\pi.
\]

Since
\[
\operatorname{spt}E_\chi
\subset
\operatorname{spt}\nabla\chi
\cup
\operatorname{spt}\Delta_g\chi,
\]
set
\[
\delta_\chi
:=
\operatorname{dist}_g
\left(
\overline{\mathcal T},
\operatorname{spt}\nabla\chi
\cup
\operatorname{spt}\Delta_g\chi
\right)
>0.
\]
Then, by \eqref{eq:heat-off-diagonal-general} and
\eqref{eq:cutoff-error}, for $m=0,1$,
\[
\begin{aligned}
\sup_{\mathcal T}
\left|
\nabla^m
\int_0^tH_{t-s}E_\chi(s)\,ds
\right|
&\leq
C
\int_0^t
(t-s)^{-(n+m+2)}
\exp\left(
-\frac{\delta_\chi^2}{4(t-s)}
\right)
(1+s^{-1/2})\,ds.
\end{aligned}
\]
Splitting the integral at $s=t/2$, one obtains
\[
\int_0^t
(t-s)^{-(n+m+2)}
\exp\left(
-\frac{\delta_\chi^2}{4(t-s)}
\right)
(1+s^{-1/2})\,ds
\leq
C
t^{-n-m-3/2}
\exp\left(
-\frac{\delta_\chi^2}{4t}
\right).
\]
Consequently, for $m = 0,1$,
\[
\sup_{\mathcal T}
\left|
\nabla^m
\int_0^tH_{t-s}E_\chi(s)\,ds
\right|
\leq
C
t^{-n-m-3/2}
\exp\left(
-\frac{\delta_\chi^2}{4t}
\right).
\]

Set
\[
\delta_0
:=
\operatorname{dist}_g
\left(
\overline{\mathcal T},
\operatorname{spt}(1-\chi)
\right)
>0.
\]
Since
\[
\left\|
(1-\chi)\sigma_v
\right\|_{L^\infty}
\leq1,
\]
the estimate \eqref{eq:heat-off-diagonal-general} gives, for $m=0,1$,
\[
\sup_{\mathcal T}
\left|
\nabla^mH_t((1-\chi)\sigma_v)
\right|
\leq
C_m
t^{-(n+m+2)}
\exp\left(
-\frac{\delta_0^2}{4t}
\right).
\]
Since $\chi=1$ on $\mathcal T$, we have
\[
\alpha(t,x)
=
\varphi(t,r(x))
\]
for $x\in\mathcal T$. Combining the preceding estimates yields
\[
\sup_{\mathcal T}
\left|
H_t\sigma_v
-
\operatorname{erf}\left(\frac{r}{2\sqrt t}\right)
\right|
\leq
C\sqrt t
\]
and
\[
\sup_{\mathcal T}
\left|
\nabla
\left(
H_t\sigma_v
-
\operatorname{erf}\left(\frac{r}{2\sqrt t}\right)
\right)
\right|
\leq
C.
\]
This completes the proof of the lemma.
\end{proof}

We now describe the zero set.  Choose $A>0$ sufficiently large.  Since
\[
\operatorname{erf}
\left(
\frac{At}{2\sqrt t}
\right)
=
\frac{A}{\sqrt\pi}\sqrt t
+
O(t^{3/2}),
\]
\eqref{eq:heat-sign-model-comparison} implies, after first fixing $A$
large enough and then taking $t$ sufficiently small, that
\[
H_t\sigma_v\bigl(\exp_q(-At\,\nu_q)\bigr)<-c\sqrt t,
\qquad
H_t\sigma_v\bigl(\exp_q(At\,\nu_q)\bigr)>c\sqrt t
\]
uniformly along the relevant portion of $\Gamma_v$.  Thus every normal
fiber contains a zero between $r=-At$ and $r=At$.

On the same region,
\[
\partial_r
\operatorname{erf}
\left(
\frac{r}{2\sqrt t}
\right)
=
\frac{1}{\sqrt{\pi t}}
\exp\left(
-\frac{r^2}{4t}
\right).
\]
For $|r|\leq At$ the right-hand side is bounded below by
$c/\sqrt t$.  Hence
\eqref{eq:heat-sign-gradient-comparison} gives, after decreasing $t$,
\begin{equation}
\label{eq:heat-normal-transversality}
\partial_rH_t\sigma_v
\geq
\frac{c}{\sqrt t}.
\end{equation}
Therefore $H_t\sigma_v$ is strictly increasing on every normal fiber,
and the zero on each fiber is unique.

Moreover, if $|r(x)|\geq At$ and $x\in\mathcal T$, then
monotonicity of the error function and
\eqref{eq:heat-sign-model-comparison} imply that $H_t\sigma_v$ does not vanish at such $x$. Now let $x\in(M\setminus U_\tau)\setminus\mathcal T$. Then
\[
\operatorname{dist}_g\bigl(x,\Gamma_v\bigr)\geq\delta>0.
\]
Set
\[
f_x(y):=\sigma_v(y)-\sigma_v(x).
\]
Since heat flow preserves constants,
\[
H_t\sigma_v(x)-\sigma_v(x)=H_tf_x(x).
\]
Note that
\[
\|f_x\|_{L^\infty}\leq2.
\]
Moreover, $f_x$ is supported on the opposite side of $\Gamma_v$, so
\[
\operatorname{dist}_g\bigl(x,\operatorname{spt}f_x\bigr)\geq\delta.
\]
Hence \eqref{eq:heat-off-diagonal-general} implies
\[
|H_t\sigma_v(x)-\sigma_v(x)|
=
|H_tf_x(x)|
\leq
Ct^{-(n+2)}
\exp\left(
-\frac{\delta^2}{4t}
\right).
\]
Thus, for sufficiently small $t$, $H_t\sigma_v$ has the same sign as
$\sigma_v$ there. Therefore, $H_t\sigma_v$ has no zeros outside $\{x\in\mathcal T: |r(x)|\leq At\}$.

By \eqref{eq:heat-off-diagonal}, the conclusions about $H_t\sigma_v$ hold for $F_t(\xi_k)$ as well for $t$ sufficiently small. Consequently, uniformly in $k$,
every relevant normal fiber contains exactly one nondegenerate zero of
$F_t(\xi_k)$, and there are no other zeros in $M\setminus U_\tau$.
Then
\[
\{F_t(\xi_k)=0\}\cap(M\setminus U_\tau)
\]
is a smooth normal graph over a subset of
\[
\Gamma_v\cap(M\setminus U_\tau).
\]

If $r_{t,k}$ denotes its graph function, then as was shown above
\[
\|r_{t,k}\|_{L^\infty}\leq At.
\]
In the tubular coordinates, differentiating
\[
\widehat F_{t,k}
\bigl(
\exp_q(r_{t,k}(q)\nu_q)
\bigr)
=
0
\]
in a direction tangent to $\Gamma_v$, and using
\eqref{eq:heat-sign-gradient-comparison},
\eqref{eq:heat-off-diagonal}, and the preceding lower bound for the
normal derivative, gives
\[
\|\nabla r_{t,k}\|_{L^\infty}
\leq
C\sqrt t.
\]
Hence
\begin{equation}
\label{eq:F-graph-small}
\sup_k
\|r_{t,k}\|_{C^1}
=
o_t(1).
\end{equation}

Finally, for this fixed $t$,
\[
\psi_k-F_t(\xi_k)\longrightarrow0
\]
in $C^\infty(M)$.  Thus, the above conclusions remain valid for $\psi_k$.  Consequently,
\[
\{\psi_k=0\}\cap(M\setminus U_\tau)
\]
is a normal graph over a subset of
$\Gamma_v\setminus D_\tau$, and its graph function satisfies
\[
\limsup_{k\to\infty}\|\tilde r_{t,k}\|_{C^1}
=
o_t(1).
\]
By the area formula,
\begin{equation}
\label{eq:outside-contribution}
\limsup_{k\to\infty}
\cH^{n}
\bigl(
\{\psi_k=0\}\cap(M\setminus U_\tau)
\bigr)
\leq
\cH^{n}
\bigl(
\Gamma_v\cap(M\setminus U_\tau)
\bigr)
+
o_t(1).
\end{equation}

\medskip
\noindent
\textbf{Step 3 (b): the Hausdorff measure estimate on $U_\tau$.}
We first prove the following local heat-flow estimate.

\begin{lemma}
\label{lem:stable-heat-rounding}
Let $\mathscr A$ be a family of open finite-perimeter subsets of
$M$.  Suppose that there are $C_0,r_0>0$ such that
\begin{equation}
\label{eq:uniform-boundary-density}
\cH^{n}(\partial A\cap B_r(x))
\leq C_0r^{n}
\end{equation}
for every $A\in\mathscr A$, $x\in M$, and $0<r<r_0$.  Put
\[
\vartheta_A:=1-2\chi_A.
\]
There is $R_0>0$ with the following property.  For every fixed
$R>0$, there are an integer $m$, constants
$C,\varepsilon_0>0$, and $t_0>0$ such that, if
$A\in\mathscr A$, $0<t<t_0$, and $\psi\in C^\infty(M)$ satisfies
\begin{equation}
\label{eq:scaled-heat-perturbation}
\sum_{\ell=0}^m
t^{\ell/2}
\|\nabla^\ell(\psi-H_t\vartheta_A)\|_{L^\infty(M)}
<\varepsilon_0,
\end{equation}
then
\begin{equation}
\label{eq:local-heat-nodal-area}
\cH^{n}
\bigl(
\{\psi=0\}\cap B_{R\sqrt t}(x)
\bigr)
\leq Ct^{n/2}
\end{equation}
for every $x\in\partial A$, and
\begin{equation}
\label{eq:heat-zero-localization}
\{\psi=0\}
\subset
\{x\in M:d_g(x,\partial A)<R_0\sqrt t\}.
\end{equation}
\end{lemma}

\begin{proof}
Fix $R>0$, and choose
\[
0<\rho<
\frac14\min\{r_0,\operatorname{inj}(M,g)\}.
\]
For $x\in M$, let $\iota_x:\mathbb R^{n+1}\longrightarrow T_xM$
be an orthonormal isometry. For $0<t<1$, define
\[
\Phi_{x,t}(z)
:=
\exp_x(\sqrt t\,\iota_xz),
\qquad
|z|<\frac{\rho}{\sqrt t},
\]
and
\[
g_{x,t}:=
t^{-1}\Phi_{x,t}^*g.
\]
Then
\begin{equation}
\label{eq:rescaled-metric-limit}
g_{x,t}\longrightarrow g_{\mathrm{eucl}}
\end{equation}
in $C^\infty_{\mathrm{loc}}(\mathbb R^{n+1})$ as $t\downarrow0$, uniformly in $x\in M$.

For $A\in\mathscr A$, let
\[
\widehat A_{x,t}
:=
\Phi_{x,t}^{-1}(A)
\]
and
\[
\widehat\vartheta_{A,x,t}
:=
\vartheta_A\circ\Phi_{x,t}
=
1-2\chi_{\widehat A_{x,t}}.
\]
For every fixed $L>0$, \eqref{eq:rescaled-metric-limit} and
\eqref{eq:uniform-boundary-density} imply that, for all sufficiently
small $t>0$,
\begin{equation}
\label{eq:rescaled-boundary-density}
\operatorname{Per}(\widehat A_{x,t};B_s(z))
\leq
C_1s^{n}
\end{equation}
whenever
\[
B_{2s}(z)\subset B_L(0),
\]
where $C_1=C_1(M,g,C_0)$ is independent of $r_0$ and $L$.

Let $\mathscr T_{C_1}$ be the set of all
\[
\vartheta=1-2\chi_E
\in \operatorname{BV}_{\mathrm{loc}}(\mathbb R^{n+1})
\]
such that
\begin{equation}
\label{eq:tangent-BV-bound}
\operatorname{Per}(E;B_s(z))
\leq
C_1s^{n}
\end{equation}
for every $z\in\mathbb R^{n+1}$ and $s>0$. By the BV compactness theorem and lower semicontinuity of
perimeter (see Ambrosio--Fusco--Pallara \cite{AFP}), $\mathscr T_{C_1}$ is compact in
$L^1_{\mathrm{loc}}(\mathbb R^{n+1})$.

For $z\in\mathbb R^{n+1},$ let
\[
G(z):=(4\pi)^{-(n+1)/2}e^{-|z|^2/4}.
\]
Set
\[
\mathscr U_{C_1}
:=
\{G*\vartheta:\vartheta\in\mathscr T_{C_1}\}.
\]
The set $\mathscr U_{C_1}$ is compact in
$C^\infty_{\mathrm{loc}}(\mathbb R^{n+1})$.  Indeed, let $u_k=G*\vartheta_k,$
where $\vartheta_k\in\mathscr T_{C_1}.$ By the compactness of $\mathscr T_{C_1}$ in
$L^1_{\mathrm{loc}}(\mathbb R^{n+1})$, after passing to a subsequence we
may assume that $\vartheta_k\longrightarrow\vartheta$
in $L^1_{\mathrm{loc}}(\mathbb R^{n+1})$ for some $\vartheta\in\mathscr T_{C_1}$.  Fix a compact set
$K\subset\mathbb R^{n+1}$ and a multi-index $\alpha$.  Then
\[
\partial^\alpha u_k-\partial^\alpha(G*\vartheta)
=
\partial^\alpha G*(\vartheta_k-\vartheta).
\]
For $L>0$,
\[
\sup_{x\in K}
\left|
\int_{B_L(0)}
\partial^\alpha G(x-y)
(\vartheta_k-\vartheta)(y)\,dy
\right|
\leq
C_{K,L,\alpha}
\|\vartheta_k-\vartheta\|_{L^1(B_L(0))},
\]
which tends to zero as $k\to\infty$.  Since
$|\vartheta_k-\vartheta|\leq2$ almost everywhere,
\[
\begin{aligned}
\sup_{x\in K}
\left|
\int_{\mathbb R^{n+1}\setminus B_L(0)}
\partial^\alpha G(x-y)
(\vartheta_k-\vartheta)(y)\,dy
\right|
&\leq
2\sup_{x\in K}
\int_{\mathbb R^{n+1}\setminus B_L(0)}
|\partial^\alpha G(x-y)|\,dy,
\end{aligned}
\]
and the right-hand side tends to zero as $L\to\infty$, because every
derivative of $G$ has exponential decay. Hence, $u_k\longrightarrow G*\vartheta$
in $C^\infty_{\mathrm{loc}}(\mathbb R^{n+1})$, and $\mathscr U_{C_1}$ is compact in $C^\infty_{\mathrm{loc}}(\mathbb R^{n+1})$.

Every $u\in\mathscr U_{C_1}$ is real analytic and is not identically
zero. Indeed, if $G*\vartheta\equiv0$, then
taking Fourier transforms gives
$\widehat G\,\widehat\vartheta=0.$
The Fourier transform of the Gaussian is a strictly positive Gaussian,
and in particular $\widehat G$ has no zeros. Hence, $\vartheta=0$. This contradicts
the fact that $|\vartheta|=1$
almost everywhere. Thus, no element of $\mathscr U_{C_1}$ is identically zero.

There are an integer $q\geq1$ and $c>0$ such that
\begin{equation}
\label{eq:model-uniform-jet}
\max_{0\leq\ell\leq q}
|\nabla^\ell u(z)|
\geq c
\end{equation}
for every $u\in\mathscr U_{C_1}$ and
$z\in\overline{B_{R+1}(0)}$. Indeed, otherwise we could find
$u_k\in\mathscr U_{C_1}$ and
$z_k\in\overline{B_{R+1}(0)}$ such that
\[
\max_{0\leq\ell\leq k}
|\nabla^\ell u_k(z_k)|
\longrightarrow0.
\]
By compactness of $\mathscr U_{C_1} \times \overline{B_{R+1}(0)}$, after passing to a subsequence,
$u_k\longrightarrow u
\quad\text{in }C^\infty(\overline{B_{R+3/2}(0)}),$ and
$z_k\longrightarrow z.$ Every derivative of $u$ then vanishes at $z$, so analyticity gives
$u\equiv0$. This contradicts the fact that no element of $\mathscr U_{C_1}$ is identically zero.

Compactness also gives
\begin{equation}
\label{eq:model-upper-jet}
M_*:=
\sup_{u\in\mathscr U_{C_1}}
\|u\|_{C^{q+1}(B_{R+1}(0))}
<\infty.
\end{equation}
\textbf{Claim 1.} There are $\varepsilon_*>0$ and $C_*>0$, depending only on
$n,C_1$ and $R$, such that every
\[
w\in C^\infty(B_{R+1}(0))
\]
satisfying
\begin{equation}
\label{eq:model-neighborhood}
\inf_{u\in\mathscr U_{C_1}}
\|w-u\|_{C^{q+1}(B_{R+1}(0))}
<\varepsilon_*
\end{equation}
also satisfies
\begin{equation}
\label{eq:compact-model-nodal-bound}
\cH^{n}(\{w=0\}\cap B_R(0))
\leq C_*.
\end{equation}

\begin{proof}[Proof of Claim 1]
To prove this, first note the following elementary fact.  For every
$1\leq\ell\leq q$, there is $b_\ell=b_\ell(n)>0$ such that for
every nonzero symmetric $\ell$-tensor $T$ on $\mathbb R^{n+1}$ one
can find an orthonormal basis $e_1,\ldots,e_{n+1}$ satisfying
\[
|T[e_i,\ldots,e_i]|
\geq b_\ell |T|,
\]
for every $i=1,\ldots,n+1.$ Indeed, for $T\neq0$ the homogeneous polynomial
\[
v\longmapsto T[v,\ldots,v]
\]
does not vanish identically on the unit sphere, so one can choose an
orthonormal basis on which none of these values vanishes.  The
assertion then follows by compactness of the unit sphere in the
finite-dimensional space of symmetric $\ell$-tensors.  Set
\[
b:=\min_{1\leq\ell\leq q}b_\ell>0.
\]

Choose
\[
0<\varepsilon_*<\min\{1,c/4\}.
\]
Suppose that $w$ satisfies
\eqref{eq:model-neighborhood}, and choose
$u\in\mathscr U_{C_1}$ such that
\[
\|w-u\|_{C^{q+1}(B_{R+1}(0))}<\varepsilon_*.
\]
Let $z\in\{w=0\}\cap B_R(0)$.  Since
\[
|u(z)|<\varepsilon_*<c
\]
and \eqref{eq:model-uniform-jet} holds, there is
$1\leq\ell\leq q$ such that
\[
|\nabla^\ell u(z)|\geq c.
\]
Hence, after decreasing $\varepsilon_*$ if necessary,
\[
|\nabla^\ell w(z)|\geq \frac c2.
\]
Choose an orthonormal basis $e_1,\ldots,e_{n+1}$ as above.  Then
\[
|\partial_{e_i}^{\ell}w(z)|
\geq \frac{bc}{2},
\]
for every $i=1,\ldots,n+1.$ Moreover,
\[
\|w\|_{C^{q+1}(B_{R+1}(0))}
\leq M_*+1.
\]
By compactness, there is $0<r_*=r_*(n,q,b,c,M_*)<1/2,$
independent of $w$ and $z$, such that
\begin{equation}
\label{eq:directional-derivative-lower}
|\partial_{e_i}^{\ell}w|
\geq \frac{bc}{4}
\end{equation}
on $B_{2r_*}(z),$ for every $i=1,\ldots,n+1.$

By Rolle's theorem, every line parallel to $e_i$ intersects
\[
\{w=0\}\cap B_{2r_*}(z)
\]
in at most $\ell\leq q$ points.  Since every zero of $w$ has
vanishing order at most $q$, Lemma 3 in B\"ar \cite{Bar}
implies that $\{w=0\}$ is countably
$n$-rectifiable.  Let $\nu$ denote an approximate unit normal
to $\{w=0\}$.  Applying the area formula to the orthogonal
projection onto $e_i^\perp$, we have
\[
\int_{\{w=0\}\cap B_{r_*}(z)}
|\nu\cdot e_i|\,d\cH^{n}
\leq
C(n)q\,r_*^{n}.
\]
Summing over $i$, and using
\[
\sum_{i=1}^{n+1}|\nu\cdot e_i|\geq1,
\]
gives
\begin{equation}
\label{eq:uniform-local-model-area}
\cH^{n}
\bigl(
\{w=0\}\cap B_{r_*}(z)
\bigr)
\leq
C(n,q)r_*^{n}.
\end{equation}

Finally, cover $B_R(0)$ by finitely many balls of radius $r_*/2$.
For each such ball which meets $\{w=0\}$, choose a zero $z$ in
that ball. The ball is then contained in $B_{r_*}(z)$.  Summing
\eqref{eq:uniform-local-model-area} over this finite covering gives
\[
\cH^{n}(\{w=0\}\cap B_R(0))
\leq C_*,
\]
where $C_*=C_*(n,C_1,R)$.  This proves
\eqref{eq:compact-model-nodal-bound}.
\end{proof}
\noindent
\textbf{Claim 2.} There is $t_1>0$ such that, for every
$A\in\mathscr A$, $x\in\partial A$, and $0<t<t_1$,
\begin{equation}
\label{eq:manifold-close-to-model}
\inf_{u_*\in\mathscr U_{C_1}}
\left\|
H_t\vartheta_A\circ\Phi_{x,t}-u_*
\right\|_{C^{q+1}(B_{R+1}(0))}
<
\varepsilon_*/2.
\end{equation}

\begin{proof}[Proof of Claim 2]
Suppose, towards a contradiction, that there are
$t_k\downarrow0$, $A_k\in\mathscr A$, and
$x_k\in\partial A_k$ for which
\eqref{eq:manifold-close-to-model} fails.  Set
\[
\widehat A_k:=\widehat A_{x_k,t_k}.
\]
Set
\[
\widehat\vartheta_k
:=
1-2\chi_{\widehat A_k}.
\]
Since $\rho/\sqrt{t_k}\longrightarrow\infty$, for every $D>0$, the functions $\widehat\vartheta_k$ are defined on
$B_D(0)$ for all sufficiently large $k$.
Thus, by BV compactness (see Ambrosio--Fusco--Pallara \cite{AFP}), after passing to a subsequence,
$\widehat\vartheta_k
\longrightarrow
\vartheta_\infty
$
in $L^1_{\mathrm{loc}}(\mathbb R^{n+1}),$ where
\[
\vartheta_\infty=1-2\chi_{E_\infty}
\]
for some measurable set $E_\infty\subset\mathbb R^{n+1}$. For every $z\in\mathbb R^{n+1}$ and $s>0$,
\eqref{eq:rescaled-boundary-density} applies to $B_s(z)$ for all
sufficiently large $k$. Then lower semicontinuity of perimeter
gives
\[
P(E_\infty;B_s(z))
\leq
\liminf_{k\to\infty}
P(\widehat A_k;B_s(z))
\leq
C_1s^{n}.
\]
Thus $\vartheta_\infty\in\mathscr T_{C_1}.$
Applying \eqref{eq:rescaled-heat-convergence} with
$f_k=\vartheta_{A_k}$ gives
\[
H_{t_k}\vartheta_{A_k}\circ\Phi_{x_k,t_k}
\longrightarrow
G*\vartheta_\infty
\]
in $C^\infty_{\mathrm{loc}}(\mathbb R^{n+1})$. Since $G*\vartheta_\infty\in\mathscr U_{C_1}$, this contradicts the
failure of \eqref{eq:manifold-close-to-model}.  This proves the claim.
\end{proof}

Set $m:=q+1.$ Let
\[
w(z):=\psi(\Phi_{x,t}(z))
\]
and
\[
u(z):=H_t\vartheta_A(\Phi_{x,t}(z)).
\]
By the chain rule,
\begin{equation}
\label{eq:scaled-Cm-comparison}
\|w-u\|_{C^{q+1}(B_{R+1}(0))}
\leq
C_R
\sum_{\ell=0}^{q+1}
t^{\ell/2}
\|\nabla^\ell(\psi-H_t\vartheta_A)\|_{L^\infty(M)}
\end{equation}
for all sufficiently small $t$, where $C_R$ is independent of
$A,x,t$, and $\psi$.

Choose $\varepsilon_0>0$ so small that the right-hand side of
\eqref{eq:scaled-Cm-comparison} is less than
$\varepsilon_*/2$ whenever
\eqref{eq:scaled-heat-perturbation} holds.  Combining this with
\eqref{eq:manifold-close-to-model}, we obtain
\[
\inf_{v\in\mathscr U_{C_1}}
\|w-v\|_{C^{q+1}(B_{R+1}(0))}
<
\varepsilon_*.
\]
Therefore \eqref{eq:compact-model-nodal-bound} gives
\[
\cH^{n}_{\mathrm{eucl}}
(\{w=0\}\cap B_R(0))
\leq C_*.
\]

For $t$ sufficiently small, using \eqref{eq:rescaled-metric-limit}, we obtain
\[
\cH^{n}
\bigl(
\{\psi=0\}\cap B_{R\sqrt t}(x)
\bigr)
=
t^{n/2}
\cH^{n}_{g_{x,t}}
\bigl(
\{w=0\}\cap B_R(0)
\bigr)
\leq
C t^{n/2}.
\]
This proves \eqref{eq:local-heat-nodal-area}.

We finally prove \eqref{eq:heat-zero-localization}.  Choose $R_0>1$
so large that
\begin{equation}
\label{eq:gaussian-tail-choice}
\int_{\mathbb R^{n+1}\setminus B_{R_0}(0)}G(z)\,dz
<
\frac{1}{16}.
\end{equation}
We claim that there exists $t_2>0$ such that, for every
$A\in\mathscr A$, $0<t<t_2$, and $y\in M$ satisfying
\[
d_g(y,\partial A)\geq R_0\sqrt t,
\]
one has
\begin{equation}
\label{eq:heat-sign-away-boundary}
|H_t\vartheta_A(y)-\vartheta_A(y)|
<
\frac14.
\end{equation}

Suppose that \eqref{eq:heat-sign-away-boundary} fails.  Then there are
$t_k\downarrow0$, $A_k\in\mathscr A$, and $y_k\in M$ such that
\[
d_g(y_k,\partial A_k)\geq R_0\sqrt{t_k}
\]
and, after passing to a subsequence,
$\vartheta_{A_k}=\sigma$ a.e. on $B_{R_0\sqrt{t_k}}(y_k)$
for some fixed $\sigma\in\{-1,1\}$, while
\[
|H_{t_k}\vartheta_{A_k}(y_k)-\sigma|
\geq\frac14.
\]
Choose orthonormal isometries
\[
\iota_k:\mathbb R^{n+1}\longrightarrow T_{y_k}M
\]
and set
\[
\Phi_k(z)
:=
\exp_{y_k}(\sqrt{t_k}\,\iota_k z).
\]
By BV compactness (see Ambrosio--Fusco--Pallara \cite{AFP}), after passing to a further
subsequence,
\[
\vartheta_{A_k}\circ\Phi_k
\longrightarrow
\vartheta_\infty
\]
in $L^1_{\mathrm{loc}}(\mathbb R^{n+1}),$
where $|\vartheta_\infty|=1$ almost everywhere.  Since
$\vartheta_{A_k}\circ\Phi_k=\sigma$
a.e. on $B_{R_0}(0),$
we also have
$\vartheta_\infty=\sigma$
a.e. on $B_{R_0}(0).$
Therefore
\[
|(G*\vartheta_\infty)(0)-\sigma|
=
\left|
\int_{\mathbb R^{n+1}}
G(z)\bigl(\vartheta_\infty(-z)-\sigma\bigr)\,dz
\right|
\leq
2\int_{\mathbb R^{n+1}\setminus B_{R_0}(0)}G(z)\,dz
<
\frac18.
\]
On the other hand, \eqref{eq:rescaled-heat-convergence}, applied with
$f_k=\vartheta_{A_k}$, gives
\[
H_{t_k}\vartheta_{A_k}(y_k)
=
\bigl(H_{t_k}\vartheta_{A_k}\circ\Phi_k\bigr)(0)
\longrightarrow
(G*\vartheta_\infty)(0).
\]
Hence, for all sufficiently large $k$,
\[
|H_{t_k}\vartheta_{A_k}(y_k)-\sigma|
<
\frac14,
\]
a contradiction.  This proves \eqref{eq:heat-sign-away-boundary}.

Finally, decrease $\varepsilon_0$, if necessary, so that
\eqref{eq:scaled-heat-perturbation} implies
\[
\|\psi-H_t\vartheta_A\|_{L^\infty(M)}
<
\frac14.
\]
If $d_g(y,\partial A)\geq R_0\sqrt t$, then
\[
|\psi(y)-\vartheta_A(y)|
\leq
|\psi(y)-H_t\vartheta_A(y)|
+
|H_t\vartheta_A(y)-\vartheta_A(y)|
<
\frac12.
\]
Since $|\vartheta_A(y)|=1$, it follows that $\psi(y)\neq0$.  Therefore
\[
\{\psi=0\}
\subset
\{y\in M:d_g(y,\partial A)<R_0\sqrt t\},
\]
which proves \eqref{eq:heat-zero-localization} and completes the proof of the lemma.
\end{proof}

Set
\[
\mathscr G_\tau:=\bigcup_{i=0}^j\Gamma_{v_i}.
\]

For
\[
\epsilon=(\epsilon_0,\ldots,\epsilon_j)
\in\{-1,1\}^{j+1},
\]
set
\[
P_{\xi,\epsilon}(Y)
:=
\sum_{i=0}^j
\lambda_i(\xi)\epsilon_iY^{c(v_i)}
\]
and
\[
\mathcal R_{\xi,\epsilon}
:=
\{s\in[0,1]:P_{\xi,\epsilon}(e^s)=0\}.
\]
Since the colors of the vertices of $\tau$ are distinct,
$P_{\xi,\epsilon}$ is a nonzero polynomial of degree at most $p$.
Therefore
\begin{equation}
\label{eq:number-selector-levels}
\sum_{\epsilon\in\{-1,1\}^{j+1}}
\#\mathcal R_{\xi,\epsilon}
\leq p2^{p+1}.
\end{equation}
Note that
\begin{equation}
\label{eq:selector-boundary-containment}
\partial A_\xi
\subset
\mathscr G_\tau
\cup
\bigcup_{\epsilon\in\{-1,1\}^{j+1}}
\bigcup_{s\in\mathcal R_{\xi,\epsilon}}h^{-1}(s).
\end{equation}

We will apply Lemma \ref{lem:stable-heat-rounding} to
\[
\mathscr A_\tau:=\{A_\xi:\xi\in\tau\}.
\]
\noindent
\textbf{Claim 3.} There are $C_p>0$ and $r_\tau>0$ such that
\begin{equation}
\label{eq:selector-boundary-density}
\cH^{n}(\partial A_\xi\cap B_r(x))
\leq C_pr^{n}
\end{equation}
for every $\xi\in\tau$, $x\in M$, and $0<r<r_\tau$.
\begin{proof}[Proof of Claim 3]
By \eqref{eq:selector-boundary-containment} and
\eqref{eq:number-selector-levels}, every boundary is contained in the
union of $\mathscr G_\tau$ and at most $p2^{p+1}$ level sets of
$h$. In sufficiently small balls, each smooth piece of
$\mathscr G_\tau$ has area at most $C_n r^{n}$.  By the property (i) in
Proposition~\ref{prop:dey-almost-smooth-discretization}, at most $n+1$
distinct pieces intersect at one point. After decreasing $r_\tau$, the
same bound holds for their union. Applying Lemma 3 in B\"ar \cite{Bar}
(see also Proposition 1 in Beck--Becker-Kahn--Hanin \cite{BBKH}) to the compact family
$\{h-s:s\in[0,1]\}$, whose vanishing order is at most two, we have that
the level sets of $h$ satisfy the
same estimate uniformly in the level. This proves the claim.
\end{proof}

\noindent
\textbf{Claim 4.} There is a function
$e_\tau(r)\to0$ as $r\downarrow0$ such that, for every
$\xi\in\tau$ and all sufficiently small $r>0$, the set
\[
\partial A_\xi
\cap
\{x\in M:d_g(x,U_\tau)\leq R_0r\}
\]
can be covered by $N_\xi(r)$ balls of radius $2r$, centered at
points of $\partial A_\xi$, with
\begin{equation}
\label{eq:selector-boundary-cover}
N_\xi(r)r^{n}
\leq
C_p
\sum_{i=0}^j
\cH^{n}(\Gamma_{v_i}\cap U_\tau)
+
C_p
\sup_{s\in[0,1]}
\cH^{n}(h^{-1}(s)\cap U_\tau)
+
e_\tau(r).
\end{equation}

\begin{proof}[Proof of Claim 4]
We prove this estimate first for the pieces of $\mathscr G_\tau$.
Remove an $r$-neighborhood of the boundaries and the nontrivial
intersections of the smooth pieces.  On the remaining part inside
$U_\tau$, a maximal $r/8$-separated family has pairwise disjoint
balls of radius $r/16$.  After decreasing $r$, every
such ball contains at least $c_n r^{n}$ area of the corresponding
smooth piece.  Hence the number of these balls, multiplied by
$r^{n}$, is bounded by
\[
C_n \sum_{i=0}^j
\cH^{n}(\Gamma_{v_i}\cap U_\tau).
\]
The removed set is contained in an $r$-neighborhood of a finite union
of smooth sets of dimension at most $n-1$, and therefore it can be
covered by balls whose number, multiplied by $r^{n}$, tends to
zero.

It remains to consider the part lying in the $R_0r$-neighborhood
of $U_\tau$ but outside $U_\tau$.  Since
$\partial U_\tau$ is transverse to every smooth hypersurface piece
occurring in $\mathscr G_\tau$, compactness and the tubular
neighborhood theorem give
\[
\sum_{i=0}^j
\cH^{n}
\left(
\Gamma_{v_i}\cap
\{x:d_g(x,\partial U_\tau)<R_0r\}
\right)
\longrightarrow0
\]
as $r\downarrow0$.

We argue similarly for the level sets of $h$.  Inside $U_\tau$,
cover first the $Ar$-neighborhoods of the finitely many critical
points of $h$ by a bounded number of $r$-balls.  Here $A$ is
fixed sufficiently large, independently of the level and of $r$.
On the complement of these neighborhoods, the same covering argument applies.
In a Morse chart the second fundamental form of a
regular level is bounded by
\[
\frac{C}{d_g(\,\cdot\,,\operatorname{Crit}h)},
\]
while away from the Morse charts it is uniformly bounded.  Thus, after
first choosing $A$ large and then decreasing $r$, every ball of
radius $r/16$, centered on the level, contains at least
$c_n r^{n}$ area of that level.  Consequently
$h^{-1}(s)\cap U_\tau$ can be covered by $N_s(r)$ balls of radius
$r$ with
\begin{equation}
\label{eq:single-level-cover-step3}
N_s(r)r^{n}
\leq
C_n\cH^{n}(h^{-1}(s)\cap U_\tau)
+o_\tau(1),
\end{equation}
where the error tends to zero uniformly in $s\in[0,1]$.

The same is true in the neighborhood of $\partial U_\tau$.
Indeed, $dh\neq0$ on $\partial U_\tau$, and
$h|_{\partial U_\tau}$ is Morse.  Away from the finitely many
critical points of $h|_{\partial U_\tau}$, the level sets of $h$
meet the boundary transversely, and their
$n$-dimensional measure in a collar of width $R_0r$ tends to
zero uniformly in the level.  Near a critical point of
$h|_{\partial U_\tau}$, the same conclusion follows from the Morse
lemma.  Hence
\[
\sup_{s\in[0,1]}
\cH^{n}
\left(
h^{-1}(s)\cap
\{x:d_g(x,\partial U_\tau)<R_0r\}
\right)
\longrightarrow0.
\]

By \eqref{eq:number-selector-levels}, only a bounded number depending
on $p$ of the level sets above occur in
$\partial A_\xi$.  Combining their covers with the cover of
$\mathscr G_\tau$, discarding every ball which does not meet
$\partial A_\xi$, and recentering each remaining ball at a point of
$\partial A_\xi$, proves
\eqref{eq:selector-boundary-cover}.
\end{proof}

Set $r:=\sqrt t.$
For this fixed $t$, condition
\eqref{eq:scaled-heat-perturbation} holds with $A=A_{\xi_k}$
and $\psi=\psi_k$
for all sufficiently large $k$, because
\[
\psi_k-F_t(\xi_k)\longrightarrow0
\]
in $C^\infty(M)$.  By
\eqref{eq:heat-zero-localization},
\[
\{\psi_k=0\}\cap U_\tau
\subset
\partial A_{\xi_k}
\cap
\{x:d_g(x,U_\tau)\leq R_0r\}.
\]
Let
$B_{2r}(x_{k,1}),\ldots,B_{2r}(x_{k,N_k})$
be the covering given by
\eqref{eq:selector-boundary-cover}.  Then
\[
\{\psi_k=0\}\cap U_\tau
\subset
\bigcup_{\ell=1}^{N_k}
B_{(R_0+2)r}(x_{k,\ell}).
\]
Since every center $x_{k,\ell}$ belongs to
$\partial A_{\xi_k}$, Lemma~\ref{lem:stable-heat-rounding}, applied
with $R=R_0+2$, gives
\[
\cH^{n}
\bigl(
\{\psi_k=0\}\cap
B_{(R_0+2)r}(x_{k,\ell})
\bigr)
\leq
C_pr^{n}.
\]
Summing over $\ell$, and using
\eqref{eq:selector-boundary-cover} and
\eqref{eq:almost-smooth-local-area}, we obtain
\begin{equation}
\label{eq:inside-contribution}
\limsup_{k\to\infty}
\cH^{n}
\bigl(
\{\psi_k=0\}\cap U_\tau
\bigr)
\leq
C_p\delta+o_\tau(t),
\end{equation}
where $o_\tau(t)\to0$ as $t\downarrow0$.

Combining \eqref{eq:outside-contribution},
\eqref{eq:inside-contribution}, and
\eqref{eq:almost-smooth-global-area}, gives
\[
\limsup_{k\to\infty}
\cH^{n}(\{\psi_k=0\})
\leq
\cH^{n}
\bigl(
\Gamma_v\cap(M\setminus U_\tau)
\bigr)
+
C_p\delta
+
o_\tau(t)
\leq
L+C_p\delta+o_\tau(t).
\]
Since there are only finitely many simplices, we may choose $t>0$
so small that all the preceding small-$t$ requirements hold
simultaneously and
$\max_\tau o_\tau(t)<\kappa.$
Then
\[
\limsup_{k\to\infty}
\cH^{n}(\{\psi_k=0\})
\leq
L+C_p\delta+\kappa,
\]
which proves \eqref{eq:stable-heat-nodal-bound} and completes Step~3.

\medskip
\noindent
\textbf{Step 4: spectral projection.}
Choose $\delta,\kappa>0$ so small that
\[
C_p\delta+2\kappa<\eta,
\]
and choose $t>0$ as in Step~3 for this value of $\kappa$.

Let $P_N:L^2(M)\longrightarrow E_N(g)$
be the $L^2$-orthogonal projection.  Since
$F_t(\widetilde X)$ is compact in $C^\infty(M)$,
\begin{equation}
\label{eq:uniform-spectral-convergence}
\sup_{\xi\in\widetilde X}
\|P_NF_t(\xi)-F_t(\xi)\|_{C^m(M)}
\longrightarrow0
\end{equation}
for every $m\geq0$.

Since $F_t(\xi)\neq0$ for every $\xi$ and $\widetilde X$ is compact,
\[
\min_{\xi\in\widetilde X}\|F_t(\xi)\|_{L^2}>0.
\]
Thus, by \eqref{eq:uniform-spectral-convergence},
$P_NF_t(\xi)\neq0$
for every $\xi\in\widetilde X$ and all sufficiently large $N$.

Let $N_k\to\infty$ and choose arbitrary
$\xi_k\in\widetilde X$.  After passing to a subsequence, we may
assume that $\xi_k\longrightarrow\xi\in\widetilde X.$
Set $\psi_k:=P_{N_k}F_t(\xi_k).$
By \eqref{eq:uniform-spectral-convergence},
\[
\|\psi_k-F_t(\xi_k)\|_{C^m(M)}
\longrightarrow0
\]
for every $m$.  Hence \eqref{eq:stable-heat-nodal-bound} gives
\[
\limsup_{k\to\infty}
\cH^{n}(\{\psi_k=0\})
\leq
L+C_p\delta+\kappa.
\]
Since the sequences $N_k\to\infty$ and
$\xi_k\in\widetilde X$ were arbitrary, it follows that
\[
\limsup_{N\to\infty}
\sup_{\xi\in\widetilde X}
\cH^{n}
\bigl(
\{P_NF_t(\xi)=0\}
\bigr)
\leq
L+C_p\delta+\kappa.
\]
Therefore, for all sufficiently large $N$,
\begin{equation}
\label{eq:spectral-nodal-bound}
\sup_{\xi\in\widetilde X}
\cH^{n}
\bigl(
\{P_NF_t(\xi)=0\}
\bigr)
\leq
L+C_p\delta+2\kappa.
\end{equation}

Since $F_t$ is odd and $P_N$ is linear, the map
\[
\widetilde\Psi:\widetilde X\longrightarrow E_N(g)\setminus\{0\},
\qquad
\widetilde\Psi(\xi):=P_NF_t(\xi),
\]
is odd. It therefore defines a continuous map $\Psi:X\longrightarrow\mathbb P(E_N(g)).$
Moreover, $\Psi^*(a_N)=\alpha,$
and hence $\Psi^*(a_N^p)=\alpha^p\neq0.$
Thus $\Psi\in\mathcal N_p^N(g).$
By \eqref{eq:spectral-nodal-bound},
\[
\nu_p^N(M,g)
\leq
L+C_p\delta+2\kappa
<
L+\eta
\leq
\omega_p(M,g)+2\eta.
\]
Since $E_N(g)\subset E_{N+1}(g)$, the sequence
$\nu_p^N(M,g)$ is nonincreasing.  Therefore
\[
\lim_{N\to\infty}\nu_p^N(M,g)
\leq
\omega_p(M,g)+2\eta.
\]
Taking $\eta\downarrow0$ completes the proof of the theorem.
\end{proof}

\vspace{0.5cm}
\noindent Department of Mathematics, University of Toronto, Toronto, Canada\\
\textit{E-mail address}: \texttt{talant.talipov@mail.utoronto.ca}

\end{document}